\documentclass[a4paper,12pt]{amsart}
\usepackage{graphicx}
\usepackage[all]{xy}
\usepackage[english]{babel}

\usepackage{amsmath,amssymb}
\usepackage{enumerate}
\usepackage{amsthm}

\usepackage{amscd}

\usepackage{amsfonts}

\usepackage{mathrsfs}
\usepackage{epsfig}

\usepackage{stmaryrd}

\usepackage{layout}

\usepackage{fullpage}

\usepackage[usenames,dvipsnames]{color}

\usepackage[T1]{fontenc}

\usepackage[backref=page]{hyperref}

\DeclareMathOperator{\Pic}{Pic}

\DeclareMathOperator{\spec}{Spec}

\DeclareMathOperator{\SCH}{Sch}

\DeclareMathOperator{\rk}{rk}

\DeclareMathOperator{\Supp}{{Supp}}

\numberwithin{equation}{section}
\newcommand{\M}{\mathcal M}
\newcommand{\I}{\mathcal I}
\newcommand{\F}{\mathcal F}

\newcommand{\Z}{\mathbb Z}
\newcommand{\E}{\mathcal E}

\newcommand{\ov}{\overline}
\newcommand{\un}{\underline}
\renewcommand{\O}{\mathcal O}

\newtheorem{thm}{Theorem}[section]
\newtheorem{prop}[thm]{Proposition}
\newtheorem{lem}[thm]{Lemma}
\newtheorem{question}[thm]{Question}
\newtheorem{ex}[thm]{Example}

\newtheorem{cor}[thm]{Corollary}
\newtheorem{fact}[thm]{Fact}
\newtheorem{convention}[thm]{}

\newtheorem{theoremalpha}{Theorem}

\theoremstyle{definition}
\newtheorem{defi}[thm]{Definition}

\theoremstyle{remark}
\newtheorem{rem}[thm]{Remark}

\def\Y{\mathcal Y}

\newcommand{\mgnb}{\overline{\mathcal M}_{g,n}}

\newcommand{\Mg}{\overline{\mathcal M}_{g}}
\newcommand{\MgP}{\overline{\mathcal M}_{g,A}}
\newcommand{\MgPsm}{\mathcal M_{g,A}}
\newcommand{\CgP}{\overline{\mathcal C}_{g,A}}
\newcommand{\CgPx}{\overline{\mathcal C}_{g,A\cup\{x\}}}
\newcommand{\MgPx}{\overline{\mathcal M}_{g,A\cup \{ x \}}}

\newcommand{\Di}{\mathcal D_{irr}}
\newcommand{\Da}{\mathcal D_{b,B}}

\newcommand{\MA}{\mathfrak M_A}

\newcommand{\CA}{{\mathfrak C}_{A}}

\newcommand{\JMP}{{\mathfrak J}_{\mathfrak M_A}}
\newcommand{\IM}{{\mathfrak I}}

\newcommand{\Jx}{{\mathfrak J}_{\mathfrak M_{A\cup\{x\}}}}

\newcommand{\JacMP}{{\mathfrak Jac}_{\mathfrak M_A}}

\newcommand{\Jacx}{{\mathfrak Jac}_{\mathfrak M_{A\cup\{x\}}}}

\newcommand{\JgA}{\overline J_{g,A}}

\newcommand{\JgPx}{\ov{\mathcal J}_{g,A\cup \{ x \}}}
\newcommand{\JgP}{\ov{\mathcal J}_{g,A}}
\newcommand{\Jgn}{\ov{\mathcal J}_{g,n}}
\newcommand{\Jac}{\mathcal J_{g,A}}
\newcommand{\Jg}{\mathcal J_{g}}
\newcommand{\UgP}{\ov{\mathcal U}_{g,A}}

\newcommand{\di}{\delta_{irr}}
\newcommand{\da}{\delta_{b,B}}
\newcommand{\daX}{\delta_{b,B}(X)}
\newcommand{\daY}{\delta_{b,B}^Y}

\begin{document}

 \title[Compactifications of the universal Jacobian]{Compactifications of the universal Jacobian over curves with marked points}
\author{Margarida Melo}
\address{Dipartimento di Matematica e Fisica, Universit\`a Roma Tre
Largo San L. Murialdo 1 - 00146 Rome (Italy)}
\address{CMUC and Mathematics Department of the University of Coimbra
Apartado 3008, EC Santa Cruz 3001 - 501 Coimbra (Portugal)}

\begin{abstract}
We construct modular compactifications of the universal Jacobian stack over the moduli stack of reduced curves with marked points depending on stability parameters obtained out of fixing a vector bundle on the universal curve.
When restricted to the locus of stable marked curves, our compactifications are Deligne-Mumford irreducible smooth stacks endowed with projective moduli spaces and, following Esteves approach to the construction of fine compactifications of Jacobians,
they parametrize torsion-free rank-1 simple sheaves satisfying a stability condition with respect to the fixed vector bundle. We also study a number of properties of our compactifications as the existence of forgetful and clutching morphisms and as well of sections from the moduli stack of stable curves with marked points. We conclude by indicating a number of different possible applications for our constructions.
\end{abstract}



\maketitle

\tableofcontents

\section{Introduction}


Let $A$ be a finite set of cardinality $n:=|A|$ and $g$ an  integer such that $2g-2+n>0$. Denote with $\MgPsm$ the moduli stack of smooth curves of genus $g$ endowed with (distinct) markings parametrized by the elements of the set $A$. The universal Jacobian stack of degree $d$ over $\MgPsm$ is the stack $\Jac^d$ parametrizing elements of $\MgPsm$ together with a line bundle of degree $d$ and it comes endowed with a forgetful morphism $\pi:\Jac^d\to \MgPsm$ given by forgetting the line bundle. Both $\MgPsm$ and $\Jac^d$ are smooth and irreducible Deligne-Mumford stacks of dimension  $3g-3+n$ and $4g-3+n$, respectively; however they are not proper.

The principal aim of the present paper is to describe different ways to complete the square
\begin{equation}\label{E:aim}
\xymatrix{
\Jac^d \ar[r]\ar[d]& {\star}\ar[d]^{\pi}\\
\MgPsm \ar[r] & {\MgP}
}
\end{equation}
to be a cartesian diagram, where $\MgP$ is the Deligne-Mumford compactification of $\MgPsm$ by stable pointed curves and $\star$ is a proper Deligne-Mumford stack endowed with a modular description and admitting a proper map $\pi$ onto $\MgP$.

Whereas for $\MgPsm$ the Deligne-Mumford compactification $\MgP$ seems to be, in many ways, a {\em God given} modular compactification, for the universal Jacobian there is no analogous construction. In fact, the problem of compactifying the universal jacobian of a single singular curve is already quite intricate and, for curves with several irreducible components, it has many different solutions, depending typically on a stability parameter one fixes on the curve (see e.g. the introduction of \cite{est1} and the references therein for an account on different compactifications of the Jacobian variety of singular curves). In order to get a compactification of the universal Jacobian over $\MgP$, one has to, moreover, consider a stability parameter that varies {\em continuously} with the curves.

A modular compactification $\overline{\mathcal P}_{g,A}^d$ giving a commutative completion of the square \eqref{E:aim} was described by the author in \cite{melo}. The construction in loc. cit. was obtained via an inductive procedure on the number of points, starting with Caporaso's compactification of $\Jg^d$ over $\Mg$, constructed in \cite{cap} via GIT, and then by setting $\overline{\mathcal P}_{g,A\cup\{x\}}^d$ to be the universal family over $\overline{\mathcal P}_{g,A}^d$.
Even if this was a natural approach to pursue, the resulting compactification has some drawbacks.
On the one hand, $\overline{\mathcal P}_{g,A}^d$ is not, in general, a Deligne-Mumford stack. In fact, it coincides with Caporaso's compactification  for $A=\emptyset$, which is Deligne-Mumford if and only if   $g.c.d.(d-g+1,2g-2)=1$ (the same holds by construction in the general case when $A\neq \emptyset$). On the other hand, as $\overline{\mathcal P}_{g,A}^d$ was constructed as a universal family building up on Caporaso's compactification, it did not allow for different modular descriptions depending upon stability conditions. In particular, $\overline{\mathcal P}_{g,A}^d$ is not endowed with a number of geometrically meaningful properties, as the existence of sections from $\MgP$, clutching morphisms and others, which one could hope to apply e.g. to the computation of {\em tautological classes} in theses spaces as well as their pullbacks to the moduli space of stable curves.

The approach we will follow in the present paper is therefore different and consists roughly on  following Esteves approach in his construction of fine compactifications for families of reduced curves  in \cite{est1} via torsion-free rank-1 \textit{simple} sheaves satisfying a stability condition with respect to a vector bundle in the family. Actually, in the first part of the paper we start by considering the following more general setup. Let $\MA$ be any open substack of the algebraic stack whose sections consist of reduced curves with marked smooth points indexed by the finite set $A$. Consider $\JMP$ to be the ($\mathbb G_m$-rigidification of the) stack whose sections consist of elements of $\MA$ together with a simple rank-1 torsion free sheaf.
 We start by proving the following statement in Proposition \ref{P:algebraicity}.
\begin{theoremalpha}
The stack $\JMP$ is algebraic and its natural forgetful morphism onto $\MA$ is representable and satisfies the existence part of the valuative criterion for properness. Moreover, if $\MA$ is Deligne-Mumford, $\JMP$ is Deligne-Mumford as well.
\end{theoremalpha}
In the case when $\MA$ is the stack of curves with planar singularities, we get moreover from Proposition \ref{P:loc-plan} that $\JMP$ is smooth.

Consider now the universal family over $\MA$, which is the stack $\CA$ whose sections consist of sections of $\MA$ each endowed with an extra section, and let $\E$ be any vector bundle over $\CA$ such that $\frac{\deg \E}{\rk \E}\in \mathbb Z$. If possible, consider also a section $\sigma:\MA\to\CA$ (notice that, e.g. in the case when $A=\emptyset$, such section may not exist). Consider the open substack $\JMP^{\E,ss}$ (resp. $\JMP^{\E,s}$, resp. $\JMP^{\E,\sigma}$) of $\JMP$ parametrizing families of pointed curves in $\MA$ together with a rank-1 torsion-free sheaf which is semistable (resp. stable, resp. $\sigma$-quasistable) with respect to $\E$ (see Definition \ref{semiquasistabledef}). Then, from Proposition \ref{P:JacProp}, we have that
\begin{theoremalpha}
\begin{enumerate}[(i)]
\item $\JMP^{\E,ss}$ is universally closed over $\MA$;
\item $\JMP^{\E,s}$ is separated over $\MA$;
\item $\JMP^{\E,\sigma}$ is proper over $\MA$.
\end{enumerate}
\end{theoremalpha}

In the second part of the paper we focus our attention in the case when $\MA:=\MgP$ is the stack of stable curves of given genus $g$ with markings indexed by the set $A$. For any integer $d$, set $\JgP^d$ to be the stack parametrizing families of curves in $\MgP$ together with a torsion-free rank-1 simple sheaf of (relative) degree $d$ on the family.
It follows from Theorem A above that $\JgP^d$ is a smooth and irreducible Deligne-Mumford stack of dimension $4g-3+n$ and that its natural forgetful morphism to $\MgP$ is representable and satisfies the existence part of the valuative criterion for properness (see Proposition \ref{P:UnivJac-DM}). Clutching morphisms and contraction morphisms given by forgetting sections in $\MgP$ induce analogous morphisms in this setup (see Proposition \ref{P:clutch-forget}).

 Consider now a $d$-polarisation $\E$ on $\MgP$, i.e., a vector bundle of rank $r>0$ and degree $r(d-g+1)$ on the universal family $\CgP$ of $\MgP$ (see Definition \ref{D:d-pol}), and a section $\sigma$ of the forgetful morphism $\pi:\CgP\to\MgP$. Let $\JgP^{\E,ss}$, $\JgP^{\E,s}$ and $\JgP^{\E,\sigma}$ be the open substacks of $\JgP$ parametrizing sheaves that are $\E$-semistable, $\E$-stable and $(\E,\sigma)$-quasistable, respectively. Then, from Theorem B above, we deduce that $\JgP^{\E,ss}$, $\JgP^{\E,s}$ and $\JgP^{\E,\sigma}$ are smooth and irreducible Deligne-Mumford stacks of finite type and that $\JgP^{\E,ss}$ is universally closed, $\JgP^{\E,s}$ is separated and that $\JgP^{\E,\sigma}$ is proper (see Proposition \ref{P:univJacProp}). By applying Koll\'ar's semipositivity criterion we show moreover in section \ref{S:coarse} that $\JgP^{\E,\sigma}$ admits a projective coarse moduli space.

Finally, we consider explicit polarisations $\E$ on $\MgP$ and we study a number of properties of the corresponding fine compactifications $\JgP^{\E,ss}$, $\JgP^{\E,s}$ and $\JgP^{\E,\sigma}$. More concretely,
given integer (or, more generally, rational) numbers $s, r, a_i, i\in A$ such that $\frac{s(2g-2)+\sum_{i\in A} a_i}{r}\in \Z$ and $\alpha_{(b,B)}$, where $(b,B)$ ranges through the admissible combinatorial types of boundary divisors of genus $b$ and with set of points $B\subset A$ (see \ref{S:UFPC}), set
\begin{equation}\label{E:vectorpol}
\E=\left (\omega^s(\sum_{i\in A}a_i\sigma_i )\otimes \sum_{(b,B)}\alpha_{(b,B)}\da\right )\oplus \mathcal O_{\CgP}^{r-1}.
\end{equation}
Then, for $d=\frac{s(2g-2)+\sum_{i\in A} a_i}{r}+g-1$, $\E$ is a $d$-polarisation on $\MgP$, i.e., $\E$-semistable sheaves will have, in particular, degree equal to $d$ (see Definition \ref{D:d-pol}). For particular choices of the coefficients we get back the so called ``canonical polarisation'' (see \ref{S:canonical} below), obtained by Caporaso in \cite{cap} in her construction of a compactification of the universal Picard variety over $\ov M_g$ in the case when $A=\emptyset$ and, more generally by Li-Wang when studying GIT-stability of nodal curves in \cite{li}.
We explore quite in detail the stability conditions that vector bundles of this shape induce on the curves and, for suitable choices of rational numbers, we show that they satisfy a number of properties as the compatibility with forgetful and clutching morphisms or the existence of generalized Abel-Jacobi maps from $\MgP$ to $\JgP$.


There is a number of quite interesting applications one might try to get to by using the constructions obtained in the present paper and we finish by indicating some of them.
To start with, the understanding of some geometrical cycle classes in $\JgP$ and their intersection theoretical properties can be applied  to compute cycle classes of interest in $\MgP$. It has been suggested by R. Hain that the existence of a compactification of the universal Jacobian variety admitting a number of specified sections from $\MgP$ can be used to compute an extension of the so-called double ramification cycle to $\MgP$, giving an answer to the so-called Eliashberg problem (see \ref{S:Eliashberg}). R. Hain himself applied this strategy to compute a partial extension to the locus of curves of compact type. Indeed, over these curves, there is no need to compactify the Jacobian as one can work with the generalised Jacobian, which parametrizes line bundles with multidegree  $\underline 0=(0,\dots, 0)$, which is already compact. The computation has been taken further by   Bashar Dudin in \cite{dudin} who, actually by using the existence of $\JgP^{\E,\sigma}$ for a suitable choice of $\E$ and $\sigma$, gave an answer to the above problem in a wider locus (including in particular all tree-like curves). The argument in loc. cit. is currently being worked out in order to further push the computation over the whole $\MgP$.
A formula for a possibly different extension of the double ramification cycle has been conjectured by Pixton and recently proved to be true by work of Janda, Pandharipande, Pixton and Zvonkine in \cite{DRC}. The two extensions coincide, as expected, over the locus of curves of compact type by work of Marcus and Wise in \cite{MW}.
It would be interesting to understand the relation of this extension with the one yield by our construction in a wider locus.

In \cite{GJV}, Goulden, Jackson and Vakil conjectured that there should exist a suitable compactification of the universal Jacobian stack over $\MgP$ together with special cycle classes whose integrals could be used to compute the so-called one-part double Hurwitz numbers, counting coverings of $\mathbb P^1$ with specified ramification over $0$ and total ramification over $\infty$, yielding a generalization of the so-called ELSV formula (see \ref{ELSV}). We intend to investigate if such a formula can be proved over $\JgP^{\E,\sigma}$, for suitable choices of $\E$ and $\sigma$.
This again depends upon the understanding of intersection theoretic properties of different algebro-geometric cycles in our construction.

A yet another possible application was suggested to us by Barbara Fantechi and consists of using the geometry of $\JgP$ in order to set up a GW-invariant theory with target the (non Deligne-Mumford) classifying stack $B\mathbb G_m$. In fact, being a parameter space for curves together with line bundles, one can give a modular interpretation of $\JgP$ as a compactification of the space of maps with target $B\mathbb G_m$ (see \ref{GW}). The theory of Gromov-Witten invariants with target a Deligne-mumford stack was developed in the algebro-geometric setting by Abramovich, Graber and Vistoli in \cite{AGV} and a gauge-theoretical approach to define GW-invariants with target $B\mathbb G_m$ was studied by Frenkel, Telemen and Tolland in \cite{FTT}. However, at least to our knowledge, there is at the moment no construction of algebro-geometric invariants on spaces of maps with target $B\mathbb G_m$. We expect to explore the application of our construction to this problem in a near future.

Finally, notice that the fact that we get universal fine compactified Jacobians can be related to the existence of universal N\'eron models for Jacobians of curves with marked points. Indeed, this has been made clear in the recent work \cite{meloneron}, where the author shows that, for a given polarization $\E$ and section $\sigma:\MA\to \CA$, the universal fine compactified Jacobian stack  $\JMP^{\E,\sigma}\to \MA$ over the moduli stack of pointed reduced curves yields universal N\'eron models in the following sense: given a family of reduced pointed curves $(f:X\to B;\sigma_1,\dots,\sigma_n)$ over a Dedekind scheme $B$ such that $f$ is smooth over a dense open subset $U\subseteq B$ and such that the total space of the family $X$ is regular, then the pullback of $\JMP^{\E,\sigma}\to \MA$ over the moduli map $\mu_f:B\to \MA$ is isomorphic to the N\'eron model of $J(X_U)$ over $B$. This is also related to recent work by D. Holmes and by A. Chiodo in \cite{holmes} and \cite{chiodo}, respectively.
For more details on this application check \cite{meloneron}.

While this paper was being finished, Kass and Pagani announced in \cite{KP} an alternative construction of a compactified universal Jacobian stack over $\MgP$ for degree $d=g-1$ and in the case when $A\neq \emptyset$. Their construction relies on the choice of a stability vector that the authors describe combinatorially over $\MgP^{(0)}$, where $\MgP^{(0)}$ denotes the locus in $\MgP$ parametrizing tree-like curves.
We show in Proposition \ref{P:compare} that, when restricting our compactified universal Jacobians $\JgP^{E,ss}$ to that locus, for suitable choices of coefficients $s, a_i$ and $\alpha_{(b,B)}$, one can  recover all the different compactifications they obtain by varying the stability parameters. The authors in loc. cit. are interested in computing the pullback to $\MgP$ of  ``compactified theta divisors'' and for that they write down wall-crossing formulas for the behaviour of those classes when the stability parameters change. It would be very interesting to investigate if we could describe similar phenomena for our compactifications as the coefficients in \eqref{E:vectorpol} vary, namely by describing a chamber structure along with wall crossing formulas, for other classes of interest.

\subsection*{Outline of the paper}

 In section \ref{S:def-not} we introduce the main characters of this work, namely curves, simple torsion free sheaves on curves, polarisations along with the stability conditions they impose. We also recall some facts concerning Esteves' compactified Jacobians for (families of) reduced curves and we summarise the main properties of the moduli stack of pointed stable curves we will need in the sequel.

 In section \ref{S:simple} we consider the more general situation of the present paper, namely the stack of simple torsion-free rank 1 sheaves over an algebraic stack of marked reduced curves. We start by proving Theorem A, which states the main general properties of this stack. We proceed by considering polarisations for these stacks in order to get separated quotients of the previous stack consisting of semi(stable) sheaves with respect to the polarisation. We then prove a number of properties of these stacks that we summarise in Theorem B.

 In section \ref{S:Mgn}, which is the main part of this paper, we focus in the special case of the moduli stack of stable marked curves. We get strengthenings of Theorems A and B in this case in Propositions \ref{P:UnivJac-DM} and \ref{P:univJacProp}, respectively. We further prove that fine universal compactified Jacobians admit a projective coarse moduli space in \S \ref{S:coarse}.
In \S \ref{S:expol} we consider explicit polarisations of the form \eqref{E:vectorpol} and we show that their associated universal compactified Jacobian stacks satisfy a number of properties. For instance, we show how to modify the polarisations in order to get the existence of forgetful morphisms, we give sufficient conditions for the compatibility with clutching morphisms and we exhibit polarisations for which the associated universal compactified Jacobian stack admits natural sections from $\MgP$.
 We finish by indicating a number of possible future applications of our constructions to a number of open questions  in \S
 \ref{S:applications}.

\section*{Acknowledgements}

I wish to thank Eduardo Esteves for many enlightening explanations on his construction of compactified Jacobians and for sharing with me many valuable ideas and suggestions and Lucia Caporaso for first guiding me into this this question. Thanks are also due to Alberto Lopez and Filippo Viviani for sharing with me many ideas at an earlier stage of this paper and  to Jesse Kass for many valuable comments on an earlier version of this paper. I wish also to thank Bashar Dudin, Barbara Fantechi and Filippo Viviani for many discussions regarding  different applications of these constructions as well as suggestions for future work.

The author was partially supported by the FCT (Portugal) projects \textit{Geometria de espa\c cos de moduli de curvas e variedades abelianas} (EXPL/MAT-GEO/1168/2013) and \textit{Comunidade Portuguesa de Geometria Alg\'ebrica} (PTDC/MAT-GEO/0675/2012) and by a Rita Levi Montalcini Grant from the italian government.

\section{Definitions and notations}\label{S:def-not}

\subsection{Stability conditions for sheaves on curves}

\subsubsection{Curves}

 Let $k$ denote an algebraically closed field (of arbitrary characteristic).

\begin{convention}\label{N:curves}
A \textbf{curve}  is a \emph{reduced} projective scheme over $k$ of pure dimension $1$, which we will assume to be  connected.

We denote by $\gamma(X)$, or simply by $\gamma$ if the curve $X$ is clear from the context, the number of irreducible components of $X$.

Unless otherwise stated, by $genus$ of a curve we will always mean the \emph{arithmetic genus}  $p_a(X)$ of $X$, i.e., a curve $X$ is said to have (arithmetic) genus $g$ if
$$g=p_a(X):=1-\chi(\O_X)=1-h^0(X,\O_X)+h^1(X, \O_X). $$
\end{convention}

\begin{convention} A \textbf{subcurve} $Z$ of a curve $X$ is a closed $k$-scheme $Z \subseteq X$ that is reduced  and of pure dimension $1$.  We say that a subcurve $Z\subseteq X$ is non-trivial if $Z\neq \emptyset, X$.
	
Given two subcurves $Z$ and $W$ of $X$ without common irreducible components, we denote by $Z\cap W$ the $0$-dimensional subscheme of $X$ obtained as the
scheme-theoretic intersection of $Z$ and $W$ and we denote by $|Z\cap W|$ its length.

Given a subcurve $Z\subseteq X$, we denote by $Z^c:=\ov{X\setminus Z}$ the \textbf{complementary subcurve} of $Z$ and we set $k_Z=k_{Z^c}:=|Z\cap Z^c|$.
 \end{convention}

\begin{convention}
A curve $X$ is called \textbf{Gorenstein} if its dualizing sheaf $\omega_X$ is a line bundle.

Given a subcurve $Z\subseteq X$ of a Gorenstein curve, we denote by $w_Z$ the degree of $\omega_X$ in $Z$, i.e., $w_Z:=\deg_Y\omega_X$.

\end{convention}

\begin{convention}
A curve $X$ has \textbf{locally planar singularities at $p\in X$} if  the completion
$\widehat{\mathcal O}_{X,p}$ of the local ring of $X$ at $p$ has embedded dimension two, or equivalently if it can be written
as
$$\widehat{\mathcal O}_{X,p}=k[[x,y]]/(f),$$
for a reduced series $f=f(x,y)\in k[[x,y]]$.
A curve $X$ has locally planar singularities if $X$ has locally planar singularities at every $p\in X$.
It is well known that a curve with locally planar singularities is Gorenstein.
\end{convention}

\begin{convention}\label{N:Jac-gen}
Given a curve $X$, the \textbf{generalized Jacobian} of $X$, denoted by $J(X)$ or by $\Pic^{\un 0}(X)$,
is the algebraic group whose $k$-valued points are line bundles on $X$ of multidegree $\un 0$ (i.e. having
degree $0$ on each irreducible component of $X$) together with the multiplication given by the tensor product.
\end{convention}

\subsubsection{Sheaves}\label{semistabledef}

We start by defining the types of sheaves we will be working with.

\begin{defi}
A coherent sheaf $I$ on a (connected) curve $X$ is said to be:
\begin{enumerate}[(i)]
\item \emph{rank-$1$} if $I$ has generic rank $1$ at every irreducible component of $X$;
\item \emph{torsion-free} if $\Supp(I)=X$ and every non-zero subsheaf $J\subseteq I$ is such that $\dim \Supp(J)=1$;
\item \emph{simple} if ${\rm End}_k(I)=k$.
\end{enumerate}
\end{defi}
\noindent Note that any line bundle on $X$ is a simple rank-$1$ torsion-free sheaf.



\noindent Given a coherent sheaf $I$ on $X$, its degree $\deg(I)$ is defined by $\deg(I):=\chi(I)-\chi(\O_X)$, where $\chi(I)$ (resp. $\chi(\O_X)$) denotes the Euler-Poincar\'e characteristic of $I$
(resp. of the trivial sheaf $\O_X$).

\noindent For a subcurve $Y$ of $X$ and a torsion-free sheaf $I$ on $X$, the restriction $I_{|Y}$ of $I$ to $Y$ is not necessarily a torsion-free sheaf on $Y$; denote by $I_Y$ the maximum torsion-free quotient of $I_{|Y}$. Notice that if $I$ is torsion-free rank-$1$ then $I_Y$ is torsion-free rank-$1$.
We let $\deg_Y (I)$ denote the degree of $I_Y$ on $Y$, that is, $$\deg_Y(I) := \chi(I_Y )-\chi(\O_Y).$$
If $X=C_1\cup\dots\cup C_\gamma$ are the irreducible components of $X$, the multidegree of $I$ is the tuple
$$(\deg_{C_1}I,\dots,\deg_{C_\gamma}I).$$

\subsubsection{Semi-stable sheaves}

Let us now recall the definitions of (semi-)stable sheaves that can be found on \cite{est1}, 1.2 and 1.4.

\begin{defi}\label{D:d-pol}
Let $X$ be a curve of genus $g$. A $d$-\textit{polarisation} on $X$ (or shortly a \textit{polarisation
}) is a vector bundle $E$ on $X$ of rank $r>0$ and degree $r(d-g+1)$.
\end{defi}


\begin{convention}
Let $E$ be a $d$-polarisation on a curve $X$ and let $Y\subseteq X$ be a subcurve of $X$. Set
$$q_Y^E:=\frac{\deg E_{|Y}}r+\frac{w_Y}2.$$
\end{convention}



\begin{defi}\label{semiquasistabledef}
Fix a $d$-polarisation $E$ on a curve $X$ and a smooth point $p\in X$.
A torsion-free rank $1$ sheaf $I$ on $X$ with $\deg(I)=d$ is said to be
\begin{enumerate}[(i)]
\item $E$-semi-stable or semi-stable with respect to $E$ if for every proper subcurve $\emptyset\neq Y\subsetneq X$, we have
\begin{equation}\label{E:stabCond}
\deg_Y(I)\geq q_Y^E-\frac{k_Y}2;
\end{equation}
\item $E$-stable or stable with respect to $E$ if inequality \eqref{E:stabCond}
holds strictly for every proper subcurve $\emptyset\neq Y\subsetneq X$;
\item $(E,p)$-quasistable, or $p$-quasistable with respect to $E$ if it is semi-stable and if strict inequality holds above whenever $p\in Y$.
\end{enumerate}
\end{defi}

It follows immediately from Definition \ref{semiquasistabledef} that the notions of stability and semistability do not depend of the vector bundle $E$, but rather of its multi-slope
$\left(\frac{\deg_{C_1}E}{r}, \dots, \frac{\deg_{C_\gamma}E}{r}\right),$
where $X=C_1\cup\dots\cup C_\gamma$ is the decomposition of $X$ in irreducible components.
Indedd, to give a polarisation on a curve $X$ it suffices to give a tuple of rational numbers $\un q=\{\un q_{C_i}\}$, one for each irreducible component $C_i$ of $X$, such that $|\un q|:=\sum_i \un q_{C_i}\in \Z$.
However, in order to deal with families of curves, which we will do, it is more convenient to deal with the vector bundle rather than with its multislope.



\begin{ex}\label{Ex:can-pol}
For any Gorenstein curve $X$ of genus $g$ and integer $d\in\mathbb Z$, consider the vector bundle
$$E^X_{can}:=\omega_X^{\otimes (d-g+1)}\oplus \mathcal O_X^{\oplus (2g-3)}$$
of degree $(2g-2)(d-g+1)$ and rank $2g-2$.
Then $E^X_{can}$ is a $d$-polarisation on $X$ and given a subcurve $Y\subset X$, we get that
$$q_Y^{E^X_{can}}=\frac{(d-g+1)(w_Y)}{2g-2}+\frac{w_Y}{2}=d\frac{w_Y}{2g-2}.$$
This implies that inequality \eqref{E:stabCond} for $E_{can}$ (at least when $X$ is a nodal curve) reduces to the well-known Gieseker-Caporaso's \textit{basic inequality} (see \cite[3.1]{cap}).
For this reason, we will call $E^X_{can}$  the \textit{canonical polarisation} of degree $d$ on $X$.
\end{ex}

\begin{defi}\label{def-int}
A polarisation $E$ is called \emph{integral} at a subcurve $Y\subseteq X$ if
$q_Z^E-\frac{k_Z}2 \in \Z$ for any connected component $Z$ of $Y$ and of $Y^c$.
A polarisation is called
\emph{general} if it is not integral at any proper subcurve $Y\subset X$.
\end{defi}

The following is Lemma 2.10 in \cite{MRV}.

\begin{fact}\label{L:nondeg}
\noindent
\begin{enumerate}[(i)]
\item If $I$ is stable with respect to a polarisation $E$ on $X$ then $I$ is simple.
\item For a polarisation $E$ on $X$, the following conditions are equivalent
\begin{enumerate}[(a)]
\item $E$ is general.
\item Every rank-1 torsion-free sheaf which is $E$-semistable is also $E$-stable.
\item Every simple rank-1 torsion-free sheaf which is $E$-semistable is also $E$-stable.
\item Every line bundle which is $E$-semistable is also $E$-stable.
\end{enumerate}
\end{enumerate}
\end{fact}

\begin{rem}
For any Deligne-Mumford stable curve $X$ of fixed genus $g$, consider $E^X_{can}$, the canonical polarisation of degree $d$ described in Example \ref{Ex:can-pol} above.
Then if follows from \cite[Prop.6.2, Lem. 6.3]{cap} that $E^X_{can}$ is general for every $X\in \ov M_g$ if and only if $(d-g+1,2g-2)=1$.
\end{rem}

The next result shows that different vector bundles can give rise to the same stability conditions and that stability behaves well with respect to tensoring with line bundles.

\begin{prop}\label{P:tensor}
Let $E$ be a $d$-polarisation on some genus $g$ curve $X$, let $p\in X$ be a smooth point and let $I$ be a torsion-free sheaf on $X$. Then
\begin{enumerate}[(i)]
\item $I$ is $E$-semistable (resp. stable, resp. $(E,p)$-quasistable) if and only if it is $E^{\oplus n}$-semistable (resp. stable, resp. $(E^{\oplus n},p)$-quasistable) for any integer $n\in \Z$.
\item Given a line bundle $L$ on $X$, $I$ is $E$-semistable (resp. stable, resp. $(E,p)$-quasistable) if and only if $I\otimes L$ is $E\otimes L$-semistable (resp. stable, resp. $(E\otimes L)$-quasistable).
\end{enumerate}
\end{prop}

\begin{proof}
The statement is a direct consequence of Definition \ref{semiquasistabledef}. For the first statement notice that $E^{\oplus n}$ is a $d$-polarisation on $X$ because $\rk E^{\oplus n}=n\rk E$ and $\deg E^{\oplus n}=n\deg E=nr(d-g+1)$. Moreover, given $Y\subseteq X$,
$$q^{E^{\oplus n}}_Y=\frac{\deg_Y E^{\oplus n}}{ \rk E^{\oplus n}}+\frac{w_Y}{2}=\frac{n\deg_Y E}{n\rk E}+\frac{w_Y}{2}=q^E_Y,$$
so the statement is clear.

For the second statement start by observing that $E\otimes L$ is a $d+\deg L$ polarisation on $X$ since $\deg (E\otimes L)=\deg E+r\deg L$ and $\rk (E\otimes L)=\rk E$, which matches the degree of $I\otimes L$. Moreover, given $Y\subseteq X$,
$$q^{E\otimes L}_Y=\frac{\deg_Y E+r\deg_Y L}{r}+\frac{w_Y}{2}=q^E_Y+\deg_Y L+\frac{w_Y}{2},$$
so inequality (\ref{E:stabCond}) holds (resp. strictly holds) for $I$ with respect to $E$ if and only if it holds (resp. strictly holds) for $I\otimes L$ with respect to $E\otimes L$.
The conclusion follows.
\end{proof}

\subsubsection{Families}

Let $f:X\to T$ be a flat, projective morphism whose geometric fibres are curves.

 \begin{convention}
Let $I$ be a $T$-flat coherent sheaf on $X$.
We say that $I$ is relatively torsion-free (resp. rank 1, resp. simple) over $T$ if the fiber of $I$ over $t$, $I(t)$, is torsion-free
(resp. rank 1, resp. simple) for every geometric point $t$ of $T$.
\end{convention}

Consider the functor
\begin{equation}\label{E:func-Jbar}
\ov{\mathbb J}_f^* : \{{\rm Schemes}/T\}  \to \{{\rm Sets}\}
\end{equation}
which associates to a $T$-scheme $Y$ the set of isomorphism classes of $T$-flat, coherent sheaves on $X\times _T Y$
which are relatively simple rank-$1$ torsion-free sheaves.

\begin{fact}[\cite{AK},\cite{est1}]\label{F:AK-simple}
The functor $\ov{\mathbb J}_f^*$ is represented by an algebraic space $\ov{J}_f$, locally of finite type over $T$ and it satisfies the existence part of the valuative criterion of properness.
\end{fact}

\begin{convention}
A relative polarisation on $X$ over $T$ is a vector bundle on $X$ of rank $r>0$ and relative degree $r(d-g+1)$ over $T$.
\end{convention}
\begin{convention}
A relatively torsion-free, rank 1 sheaf $I$ on $X$ over $T$ is (relatively) stable
(resp. semistable) with respect to a polarisation $E$ over $T$ if $I(t)$ is stable
(resp. semistable) with respect to $E(t)$ for every geometric
point $t$ of $T$.

\noindent Let $\sigma:T\to X$ be a section of $f$ through the $T$-smooth locus of $X$.
A relatively torsion-free, rank 1 sheaf $I$ on $X$ over $T$ is relatively $\sigma$-quasistable
with respect to $E$ over $T$ if $I(t)$ is $\sigma(t)$-quasistable with respect to $E(t)$ for every
geometric point $t$ of $T$.
\end{convention}

\begin{rem}\label{R:bounded}
Let $f:X\to T$ be a family of curves as above and assume that $T$ is Noetherian.Then  it follows directly from the definition that the family of all sheaves on $X$ that are relatively semistable over $T$ with respect to a fixed polarisation is bounded.
\end{rem}

Let $E$ be a $d$-polarisation on the family $f:X\to T$ and $\sigma:T\to X$ a smooth section of $f$. Denote by $\ov{\mathbb J}_f^{E,ss}$ (resp. $\ov{\mathbb J}_f^{E,s}$, resp. $\ov{\mathbb J}_f^{E, \sigma}$) the subfunctors of $\ov{\mathbb J}_f$ parametrizing (relatively) $E$-semistable (resp. $E$-stable, resp. $(E, \sigma)$-quasistable) sheaves. Then we have the following result due to Esteves  (see \cite[Theorem A]{est1}:

\begin{fact}\label{F:esteves}
The moduli functors of relatively $E$-semistable (resp. $E$-stable, resp. $(E, \sigma)$-quasistable) simple torsion-free $T$-flat sheaves over $X$,   $\ov{\mathbb J}_f^{E,ss}$, resp. $\ov{\mathbb J}_f^{E,s}$, resp. $\ov{\mathbb J}_f^{E, \sigma}$ are representable by algebraic spaces $\ov J_f^{E,ss}$, resp. $\ov J_f^{E,s}$, resp. $\ov J_f^{E, \sigma}$) of finite type over $T$. Moreover,
\begin{enumerate}[(i)]
\item $\ov J_f^{E,ss}$ is universally closed over $T$;
\item $\ov J_f^{E,s}$ is separated over $T$;
\item $\ov J_f^{E, \sigma}$ is proper over $T$.
\end{enumerate}
\end{fact}

\begin{rem}
It follows from \cite{est1} that if $T=k$ is an algebraically closed field, then $J_f^{E, \sigma}$ is actually a projective scheme.
\end{rem}

\subsection{The moduli stack of pointed stable curves}\label{S:pointed}

In what follows we will use the notation of \cite[Ch. 12, sec. 10, Ch. 17, sec. 3]{ACG}, for which we also refer for further details.

\subsubsection{Universal family and clutchings}\label{S:UFPC}

Let $A$ be a finite set and denote by $\MgP$ the moduli stack of $A$-pointed stable curves and by $\CgP$ the universal family over $\MgP$. Recall that sections of $\CgP$ consist of families of $A$-pointed stable curves together with an extra section with no stability assumptions (therefore possibly intersecting the other sections or the singular locus of the fibers) and that there is a natural isomorphism of stacks
\begin{equation}\label{F:univFamily}
\lambda:\CgP\rightarrow \MgPx
\end{equation}
which is also known as the stabilization morphism.

\begin{convention}
Denote by
$$\pi_x:\MgPx\cong \CgP\rightarrow \MgP$$
the natural forgetful morphism and by
$$\sigma_i:\MgP\rightarrow \MgPx\cong\CgP$$ for $i\in A$, the natural sections of $\pi_x$ given by identifying the extra section $\sigma_x$ with $\sigma_i$.
\end{convention}

The boundary $\MgP\setminus \mathcal M_{g,A}$ of $\MgP$ has a stratification in terms of irreducible divisors as follows:
\begin{itemize}
\item the divisor $\Di$,  which parametrizes the closure of the locus of irreducible curves with one node;
\item for $0\leq b\leq g$ and $B\subseteq A$, the divisors $\Da$, which parametrize the closure of the locus of pointed stable curves $(X;p_1,\dots, p_n)$ with one separating node $p$ of type $(b,B)$, i.e., such that the normalization of $X$ in $p$ is the union of two curves $X_p$ and $X_p'$ such that $X_p\in\overline{\mathcal M}_{b,B}$ and $X_p'\in\ov{\mathcal M}_{g-b,B^c}$ (where $B^c$ denotes $A\setminus B$).
\end{itemize}

Notice that $\Da=\mathcal D_{g-b, B^c}$ and that the stability condition implies that if $b=0$, resp. $b=g$, then the divisors $\Da$  are only defined, respectively if $|B|\geq 2$ or $|B^c|\geq 2$.

Denote by $N_{b,B}(X)$ the set of nodes $p$ in $X$ of type $(b,B)$.

\begin{convention}
There are clutching morphisms
\begin{equation}\label{E:chi-irr}
\xi_{irr}:\ov{\mathcal M}_{g-1,A\cup\{x,y\}}\rightarrow \ov{\M}_{g,A}
\end{equation}
and
\begin{equation}\label{E:chi-bB}
\xi_{b,B}:\ov{\mathcal M}_{b,B\cup\{x\}}\times \ov{\M}_{g-b,B^c\cup\{y\}}\rightarrow \MgP
\end{equation}
whose images are equal to $\Di$ and $\Da$, respectively.
\end{convention}

If $\xi_{irr}(f:X\to T;\sigma_1,\dots,\sigma_{|B|},\sigma_x,\sigma_y)=(\ov f:\ov X\to T;\ov\sigma_1,\dots,\ov\sigma_{|B|})$,
there is an induced $T$-morphism
\begin{equation}
\xi_{irr}^f:X\to \ov X
\end{equation}
compatible with the first $|A|$ sections which, over each geometric fiber of $X_t$ of $f$, is given by the partial normalization of $\ov X_t$ at the node introduced by the identification of $\sigma_x(t)$ with $\sigma_y(t)$.

Similarly, given $$(f_1:X_1\to T;\sigma_1,\dots,\sigma_{|B|},\sigma_x)\in \ov{\mathcal M}_{b,B\cup\{x\}}$$ and $$(f_2:X_2\to T;\sigma_1',\dots,\sigma_{|B|^c}',\sigma_y)\in \ov{\mathcal M}_{g-b,B^c\cup\{y\}}$$ with image by $\xi_{b,B}$ equal to $(\ov f:\ov X\to T;\ov \sigma_1,\dots,\ov \sigma_{|P|})\in \ov{\mathcal M}_{g,A}$,
$\xi_{b,B}$ induces a $T$-morphism
\begin{equation}
\xi_{b,B}^{f_1,f_2}:X_1\cup X_2\to \ov X
\end{equation}
with the obvious compatibility condition in the sections, which is again, fiber by fiber, given by the partial normalization of $\ov X$ at the node introduced by the identification of $\sigma_x$ with $\sigma_y$.

\subsubsection{Divisor classes on $\MgP$ and the Picard group}

Following \cite{arbcorn}, \cite{arbcorn2} and \cite{ACG}, we will now briefly consider some divisor classes on $\MgP$ in order to describe its Picard group.

We start with the classes of the components of the boundary of $\MgP$, that is, the classes of the irreducible divisors $\Di$ and $\Da$ that we have introduced in the previous section. Denote by $\di$ and $\da$ the classes of $\Di$ and $\Da$, respectively, and write $\delta$ to indicate the total class of the boundary of $\MgP$, that is, the sum of $\di$ and of all the classes $\da$.
Next, consider the forgetful morphism
$$\pi_x : \MgPx\to\MgP$$
and denote by $\omega={\omega}_{\pi_x}$ the relative dualizing sheaf. For every $i\in A$, the image of the section
$\sigma_i:\MgP\rightarrow \MgPx$ is precisely $\mathcal D_{0,\{i,x\}}$. Consider also the so called $\psi$-classes
$$\psi_i:=\sigma_i^*(\omega), \; i\in A$$
and the Hodge class $\lambda$, which is just the first Chern class of the locally free sheaf ${\pi_x}_*(\omega)$,
$$\lambda:=c_1({\pi_x}_*(\omega)).$$
Using results of \cite{harer}, Arbarello-Cornalba in \cite[Theorem 2]{arbcorn} described the Picard group of $\MgP$, as follows.

\begin{thm}[Arbarello-Cornalba]\label{T:arbcorn}
For all $g\geq 3$ and any finite set $A$, Pic$(\MgP)$ is freely generated by $\lambda$, the classes $\psi_i$, $i\in A$, and the boundary classes. For $g=2$, Pic$(\MgP)$ is still generated by these classes, but there are relations in this case.
\end{thm}

\section{The stack of simple sheaves over the stack of reduced curves}\label{S:simple}

Let $A$ be a finite set and let $\mathfrak M_A$ be the stack of $A$-marked reduced curves or any open substack of this. The fibers of $\mathfrak M_A$ over a scheme $T$ consist therefore of proper and $T$-flat morphisms of schemes
$f:X\to T$ whose fibers are reduced curves together with $n:=|A|$ smooth (and distinct) sections of $f$, $\sigma_i:T\to X$, $i\in A$ (notice that $A$ could be empty).
Recall that $\mathfrak M_A$ is known to be algebraic (see e.g. \cite[Corollary 1.5]{hall}).

Denote by $\mathfrak J_{\mathfrak M_A}^*$ the category fibered in groupoids (CFG) over $\SCH_S$ whose fibers over an $S$-scheme $T$ consist of pairs
 $(\xymatrix{
X \ar[r]_{f} & T \ar @/_/[l]_{\sigma_i}
}
, F)$
where
\begin{itemize}
\item $(f:X\to T,\sigma_{i,i\in A})\in\mathfrak M_A(T)$;
\item $F$ is a $T$-flat coherent sheaf over $X$ whose fibers over $T$ are torsion-free rank-1 simple sheaves.
\end{itemize}
Morphisms between two such pairs
$(\xymatrix{
X \ar[r]_{f} & T \ar @/_/[l]_{\sigma_i}
}
, F)$
and
$(\xymatrix{
X' \ar[r]_{f'} & T' \ar @/_/[l]_{\sigma'_i}
}
, F')$
are given by cartesian diagrams
\begin{equation}\label{D:morphism}
\xymatrix{
{X} \ar[r]^{\bar g} \ar[d]^{f} & {X'} \ar[d]_{f'}\\
{T} \ar[r]_{g}  \ar @{->} @/^/[u]^{\sigma_i} & {T'} \ar @{->} @/_/[u]_{\sigma'_{i}}
}
\end{equation}
commuting with the sections, i.e., such that $\bar g\circ\sigma_i=\sigma'_i \circ g, \forall i\in A$, together with an isomorphism
$$\beta: F\to\bar g^* F'\otimes f^*(M)$$
for some $M\in\Pic(T)$.

One checks that $\mathfrak J_{\mathfrak M_A}^*$ is a pre-stack and define $\mathfrak J_{\mathfrak M_A}$ to be its stackification. Denote by $\pi$ the natural forgetful morphism from $\mathfrak J_{\mathfrak M_A}$ to $\mathfrak M_A$.

\begin{rem}\label{R:rigidification}
Notice that $\mathfrak J_{\mathfrak M_A}$ is the $\mathbb G_m$-rigidification of the stack  $\mathfrak Jac_{\mathfrak M_A}$ whose sections coincide with those of $\mathfrak J_{\mathfrak M_A}$ and whose morphisms are as in \eqref{D:morphism} together with an isomorphism $\alpha:F\to\bar g^* F'$.
\end{rem}

The following statement shows that the results of Altman and Kleiman and of Esteves in \cite[Theorem 7.4]{AK} and \cite[Theorem 32]{est1}, respectively, hold in this more general setup.

\begin{prop}\label{P:algebraicity}
The stack $\mathfrak J_{\mathfrak M_A}$ (and $\mathfrak Jac_{\mathfrak M_A}$) is algebraic and its natural forgetful morphism onto $\mathfrak M_A$ is representable and satisfies the existence part of the valuative criterion for properness.
Moreover, if ${\mathfrak M_A}$ is Deligne-Mumford, then also $\mathfrak J_{\mathfrak M_A}$ will be Deligne-Mumford.
\end{prop}

\begin{proof}

Notice that the algebraicity of $\mathfrak J_{\mathfrak M_A}$ together with the last statement will follow by showing that $\mathfrak J_{\mathfrak M_A}$ is representable over $\mathfrak M_A$. Indeed, the base-change of a smooth atlas of ${\mathfrak M_A}$ will be an atlas for $\mathfrak J_{\mathfrak M_A}$, which will turn out to be \'etale if the original atlas of ${\mathfrak M_A}$ was \'etale. This shows that the last statement follows from our proof of the first. Moreover, the algebraicity of $\mathfrak Jac_{\mathfrak M_A}$ follows  from the fact that it is a $\mathbb G_m$-gerbe over $\mathfrak J_{\mathfrak M}$, which is algebraic.



Consider a morphism $\mu_f:T\to \mathfrak M$ and let $f:X\to T$ be the element of $\mathfrak M(T)$ associated to $\mu$. Let
$$\mathfrak J_f:=T\times_{\mathfrak M}\mathfrak J_{\mathfrak M}$$
be the base-change of $\pi$ via $\mu_f$.

It is easy to see that $\mathfrak J_f$ is a CFG over $\SCH_T$ whose sections over a $T$-scheme $U$ consist of triples
$$(X_U\to U, F, \alpha)$$
where:
\begin{itemize}
\item
\begin{equation*}
\xymatrix{
{X_U:=U\times_T X} \ar[r] \ar[d] & {X} \ar[d]^{f}\\
{U} \ar[r] & {T}
}
\end{equation*}
\item  $F$ is a $U$-flat coherent sheaf over $X_U$ whose fibers over $U$ are torsion-free rank-1 simple sheaves;
\item $\alpha:X_U\to X_U$ is a $U$-isomorphism of $X_U$.
\end{itemize}

Consider now the functor
$$\mathbb J_f^*:\SCH_T\to SET$$
that associates to a $T$-scheme $U$ the set of $T$-flat coherent sheaves over $X$ whose fibers over $T$ are torsion-free rank-1 sheaves.
It follows from the work of Altman and Kleiman in \cite{AK} that, since the fibers of $f:X\to T$ are reduced, $\mathbb J_f^*$ is a pre-sheaf and its \'etale sheafification $\mathbb J_f$ is representable by an algebraic space $\ov J_f$.

Now, it is easy to check that the CFG associated to $\mathbb J_f$ is equivalent to $\mathfrak J_f$ and the result follows from the representability of $\mathbb J_f$.

In order to check that $\pi:\mathfrak J_{\mathfrak M_A}\to {\mathfrak M_A}$ satisfies the existence part of the valuative criterion for properness, consider a morphism $\mu_f:\spec R\to \mathfrak M_A$, where $R$ is a $DVR$ with fraction field $K$ and $\mu_f$ is the moduli morphism associated to a section $f:X\to \spec R$ in $\mathfrak M_A(R)$; and assume there is a commutative diagram
$$\xymatrix{
\spec K \ar[r] \ar[d]_{i} & \mathfrak J_{\mathfrak M_A} \ar[d]^{\pi} \\
\spec R \ar[r]_{\mu_f} & \mathfrak M_A.
}
$$
We shall show that there is a morphism $h:\spec R\to \mathfrak J_{\mathfrak M_A}$ making the digram commute.
From the previous part, we know that $\spec R\times_{\mathfrak M_A} \mathfrak J_{\mathfrak M_A}=\mathfrak J_f\cong \ov J_f$, the algebraic space representing the functor $\mathbb J_f$.
Let $\iota_f:\spec K\to \ov J_f$ be the unique morphism making the following diagram commute:
$$\xymatrix{
\spec K \ar[rd]^{\iota_f} \ar@/^/[rrd] \ar@/_/[rdd]_{i}\\
& \ov J_f \ar[r]^{\ov\mu_f} \ar[d]_{\pi_f} & \mathfrak J_{\mathfrak M_A} \ar[d]^{\pi} \\
&\spec R \ar[r]_{\mu_f} & \mathfrak M_A,
}
$$
where the outwards diagram is the previous one and the inwards one is cartesian.
Now, from Fact \ref{F:AK-simple}, we know that $\ov J_f\to \spec R$ satisfies the existence part of the valuative criterion for properness.
This implies that there is a morphism $\ov h:\spec R\to \ov J_f$ such that $\ov h\circ i=\iota_f$ and $\pi_f\circ \ov h=id$. Now, for $h:=\ov\mu_f\circ \ov h$, an easy diagram chasing shows that the whole diagram commutes, so we get that there exists the desirable lifting of $\ov\mu_f$ to $\spec R$, and we conclude.

\end{proof}

\begin{rem}
The representability of $\mathfrak J_{\mathfrak M_A}$ over ${\mathfrak M_A}$ would also follow from Abramovich-Vistoli's citerion on \cite[4.4.3]{AV} as we are considering simple sheaves. However,  this criterion is stated in the case when both
the domain and the codomain are Deligne-Mumford to start with, so we can not apply it directly here.

\end{rem}


Notice that $\JMP$ is certainly not of finite type: first of all it decomposes in infinitely many components
$$\JMP=\coprod_{d\in \Z} \JMP^d$$
where $\JMP^d$ denotes the open substack of $\JMP$ parametrizing families of $A$-pointed curves endowed with relatively flat torsion-free sheaves of relative degree $d$. In the same way, we will write
$$\JacMP=\coprod_{d\in \Z} \JacMP^d$$
for the analogous decomposition of $\JacMP$ by degree.

\begin{rem}\label{R:isom}
Notice that if $A\neq \emptyset$, then for all $d,d'\in \mathbb Z$, $\JMP^d\cong \JMP^{d'}$ (analogously  $\JacMP^d\cong \JacMP^{d'}$).
In fact, it is enough to pick $j\in A$ and consider the following isomorphism
\begin{align*}
\phi_j^{d'-d}: & \hspace{.9cm}\JMP^d \hspace{.9cm}\longrightarrow \hspace{.9cm}\JMP^{d'}\\
& (\xymatrix{
X \ar[r]_{f} & T \ar @/_/[l]_{\sigma_i}
}
, F) \mapsto   (\xymatrix{
X \ar[r]_{f} & T \ar @/_/[l]_{\sigma_i}
}, F\otimes (\mathcal O_X(\sigma_j))^{\otimes(d'-d)})
\end{align*}
(we will also denote by $\phi_j^{d'-d}$ the analogous isomorphism between $\JacMP^d$ and $\JacMP^{d'}$).
\end{rem}

\subsection{The universal family over $\JMP$}

Denote by $\Jacx$ the universal family over $\JacMP$. Sections of $\Jacx$ over an $S$-scheme $T$ consist of a section of $\JacMP$ over $T$, say $(f:X\to T; \sigma_1,\dots,\sigma_{|A|};F)$ together with an extra section $\sigma:T\to X$ (with no further requirements); denote by $\tilde \pi_x:\Jacx\to \JacMP$ the natural forgetful morphism from $\Jacx$ onto $\JacMP$.
Similarly, denote by $\Jx$ the universal family over $\JMP$, and by $\pi_x:\Jx\to\JMP$ the forgetful morphism.

The universal family $\Jacx$ is endowed with a universal sheaf $\IM$ defined as follows: it assigns to a section $(f:X\to T; \sigma_1,\dots,\sigma_{|A|};F;\sigma)$ of $\Jacx$ over $T$ the sheaf $\sigma^*(F)$ on $T$,
$$\IM\left (\xymatrix{
X \ar[r]_{f} & T \ar @/_/[l]_{\sigma_i}
}
;F;\sigma\right):= \sigma^*(F).$$



Consider the following diagram, where the horizontal morphisms are the rigidification maps
$$\xymatrix{
\Jacx \ar[d]_{\tilde\pi} \ar[r]^{\nu_1} & \Jx \ar[d]^{\pi}\\
\JacMP  \ar[r]_{\nu} &\JMP.}$$

We shall now see that, if $A\neq \emptyset$, the universal sheaf $\IM$ descends to the rigidification $\JMP$.

\begin{prop}\label{P:univ-desc}
Let $A\neq \emptyset$. Then the universal sheaf $\IM$ on $\JacMP$ descends to a universal sheaf $\ov\IM$ on $\JMP$.
\end{prop}
\begin{proof}
Start by noticing that it is enough to check that the restriction of $\IM$ to each component $\JacMP^d$, where $d$ is any integer, descends to $\JMP$. We can further assume that all curves have the same arithmetic genus, which we will denote by $g$.
Moreover, as a consequence of Remark \ref{R:isom} (which we can use since $A\neq \emptyset$), it is enough to consider one specific degree, so let us prove that the universal sheaf on $\JacMP^g$ descends to $\JMP^g$.

Consider the $\mathbb G_m$-gerbe structure morphism
$$\nu:\JacMP\to\JMP.$$
As explained in \cite{MVPic}, $\nu$ induces the following exact sequence at the level of Picard groups

$$0\longrightarrow \Pic(\JMP^g)\stackrel{\nu^*}{\longrightarrow} \Pic(\JacMP^g) \stackrel{\rm res}{\longrightarrow} \Pic(B\mathbb G_m)\cong\mathbb Z .$$
We will show that the map $\rm res$ is surjective.
Using the same notation as in \cite{MVPic}, denote by $\Lambda(0,1)$ the line bundle on $\Pic(\JacMP^g)$ defined as
$$\Lambda(0,1):=d_{\tilde\pi_x}(\IM^g)$$
where $\IM^g$ is the restriction to $\Jacx$ of $\IM$ and $d_{\tilde\pi_x}$ denotes the determinant of cohomology of the morphism $\tilde\pi_x$.
Then, the same computation done in \cite[Lemma 6.2]{MVPic} shows that $\rm res(\Lambda(0,1))=1$.

\end{proof}

\begin{rem}
Notice that, as shown in \cite[Lemma 6.2]{MVPic}, the statement of Proposition \ref{P:univ-desc} does not hold in the case when $A=\emptyset$. Indeed, in this case it is shown in loc. cit. that the statement holds if and only if $(d-g+1,2g-2)=1$.
\end{rem}

\subsection{Curves with locally planar singularities}

In the following we show that if we restrict to the case of curves of given genus $g$ and with locally planar singularities $\mathfrak M_g^{lp}$ then the stack of simple sheaves over it is actually smooth and, if the degree of the sheaves is fixed, also irreducible.

\begin{prop}\label{P:loc-plan}
Let $\mathfrak M^{lp}_{g,A}$ denote the stack of $A$-pointed reduced curves of genus $g$ with locally planar singularities. Then, the stack $\mathfrak{J}_{\mathfrak M^{lp}_{g,A}}$ of sections of $\mathfrak M^{lp}_{g,A}$  endowed with rank $1$ torsion-free simple sheaves is smooth and its substack  $\mathfrak J^d_{\mathfrak M^{lp}_{g,A}}$ parametrizing sheaves with fixed degree $d\in\mathbb Z$ is also irreducible.
\end{prop}

\begin{proof}
Start by noticing that $J_{\mathfrak M^{lp}_{g,A}}$ is locally of finite presentation. Then the lifting criterion for smoothness holds also to prove the smoothness of $J_{\mathfrak M^{lp}_{g,A}}$ (see e.g. \cite[Section 2]{heinloth}).
Start by considering the case when $A=\emptyset$.
Let $R$ be a local Artinian algebra and $J\subset R$ an ideal of $R$ with $J^2=0$. We must then check that given a commutative diagram
\begin{equation*}
\xymatrix{
\spec(R/J) \ar[r] \ar[d]& J_{\mathfrak M^{lp}_{g,A}}\ar[d]\\
\spec R \ar[r] \ar@{-->}[ur]& k
}
\end{equation*}
the dotted line can be completed, i.e., the morphism $\spec(R/J)\to J_{\mathfrak M^{lp}_{g,A}}$ can be extended to a morphism $\spec R\to J_{\mathfrak M^{lp}_{g,A}}$.
Let $(X,I)$ be the image of $0$ under $\spec (R/I)\to J_{\mathfrak M^{lp}_{g,A}}$ and consider the local deformation functor Def$_{(X,I)}$ of the pair $(X,I)$, where $X$ is assumed to be a reduced curve and $I$ a rank $1$ torsion-free sheaf (not necessarily simple). Then, according to \cite{FGS} (see also \cite[Fact 3.2]{MRV}), the functor Def$_{(X,I)}$ is smooth, which implies that Def$_{(X,I)}(R)$ surjects onto Def$_{(X,I)}(R/J)$. We easily conclude that the above dotted morphism exists and we conclude.

Let us now see that the pointed case follows from the unpointed case by induction on the number $|A|$ of points in $A$. Suppose that $J_{\mathfrak M^{lp}_A}$ is smooth and consider $J_{\mathfrak M^{lp}_{A\cup \{x\}}}$. Then $J_{\mathfrak M^{lp}_{A\cup \{x\}}}$ is an open substack of $J_{\mathfrak M^{lp}_{A\cup\{x\}}}$, the universal family over $J_{\mathfrak M^{lp}_{A}}$. But if $J_{\mathfrak M^{lp}_A}$ is smooth, it follows immediately that the universal family $J_{\mathfrak M^{lp}_{A\cup\{x\}}}$ is also smooth, so $J_{\mathfrak M^{lp}_{A\cup \{x\}}}$ is smooth too.

Finally, to show that $J^d_{\mathfrak M^{lp}_A}$ is irreducible it is enough to observe that it is smooth and that the locus inside it parametrizing smooth curves is dense.

\end{proof}

\subsection{Fine compactified Jacobian stacks}\label{S:finequotients}

Following Esteves constructions of fine compactifications of Jacobians for families of reduced curves in \cite{est1}, we will now consider certain open substacks of the stack of simple sheaves in order to get universally closed, separated or proper substacks of $\JMP$ (see Fact \ref{F:esteves}). We start by discussing how to endow $\mathfrak M_A$ with a (continuous) polarisation.

\subsubsection{Polarisations on $\mathfrak M_A$}

In order to endow the moduli stack of $A$-pointed curves with a natural polarisation we have to prescribe, for each family $(f:X\to T; \sigma_1,\dots, \sigma_{|A|})\in \mathfrak M_A(T)$, a relative polarisation of $X$ over $T$ in a functorial way. We will assume for what follows that the sections of $\mathfrak M_A$ are families of curves having genus equal to $g$. We thus define, following Definition \ref{D:d-pol} for the case of a single curve, what do we mean with a $d$-polarisation on $\MgP$.

\begin{defi}
A $d$-polarisation (or just polarisation) on $\MA$ is a vector bundle $\E$ on $\mathfrak M_{A\cup\{x\}}$ of rank $r>0$ and degree $r(d-g+1)$ for some integer $d$.
\end{defi}

Let $(f:X\to T; \sigma_1,\dots, \sigma_{|A|})\in \MA(T)$ and consider the moduli morphism $\mu_f:T\to\MA$. Then by definition the pullback of $\mathfrak M_{A\cup\{x\}}\to\MA$ over $\mu$ gives back the family $X\to T$. So, using the notation of the cartesian diagram
\begin{equation}\label{N:pol-families}
\xymatrix{
X\ar[r]^{\ov\mu_f} \ar[d]_f &\mathfrak M_{A\cup\{x\}} \ar[d]\\
T\ar[r]_{\mu_f} &\MA
}
\end{equation}
we see that a vector bundle $\E$ on $\mathfrak M_{A\cup\{x\}}$ prescribes on $f:X\to T$ the vector bundle ${\ov\mu_f}^*(\E)$; denote this polarisation by $\E^f$. Equivalently, we consider the fiber of  $\mathfrak M_{A\cup\{x\}}$ over $(f:X\to T; \sigma_1,\dots, \sigma_{|A|})$, which is a family over $X$, and consider the vector bundle over $X$ that $\E$ prescribes for that family.

\begin{defi}
Let $\E$ be a polarisation on $\MA$ and $\mathcal Y$ a substack of $\MA$. Then we say that $\E$ is general on $\Y$ if, for all pointed curves $(f:X\to \spec(k),{p}_{i,i\in A})\in \mathcal Y(k)$, the polarisation $\E^f$ is general on the family $f:X\to\spec(k) $ in the sense of Definition \ref{def-int}.
\end{defi}

\subsubsection{Substacks of  Jacobians of (semi)stable sheaves}

\begin{defi}\label{D:polarisation}
Let $\E$ be a $d$-polarisation on $\MA$ and $\sigma$ a section of $\mathfrak M_{A\cup\{x\}}\to \MA$, in case it exists. We will denote with $\JMP^{\E,ss}$ (resp. $\JMP^{\E,s}$, $\JMP^{\E,\sigma}$) the substack of $\JMP$ parametrizing families in $\MA$ together with $\E$-semistable (resp. $\E$-stable, $(\E,\sigma)$-quasistable) torsion-free rank-1 simple sheaves.
\end{defi}

\begin{prop}\label{P:JacProp}
Let $\E$ be a $d$-polarisation on $\MA$
and $\sigma$ a section of $\CA\to \MA$, in case it exists.
%
Then the stacks $\JMP^{\E,\star}$  are open substacks of $\JMP$ endowed with a representable
morphism of finite type $\pi:\JMP^{\E,\star}\to \MA$, where $\star=ss,s$ or $\sigma$, such that:
\begin{enumerate}
\item $\pi:\JMP^{\E,ss}\to\MA$ is universally closed;
\item $\pi:\JMP^{\E,s}\to\MA$ is separated;
\item $\pi:\JMP^{\E,\sigma}\to \MA$ is proper.
\end{enumerate}
\end{prop}
\begin{proof}

The fact that $\JMP^{\E,ss}$, $\JMP^{\E,s}$ and $\JMP^{\E,\tau}$ are open substacks of $\JMP$ follows by semicontinuity from the cohomological characterizations given by Esteves in \cite[Theorem 11 and Lemma 12]{est1} while the fact that $\pi$ is of finite type in each case follows from Remark \ref{R:bounded} above.

Now the fact that $\JMP^{\E,ss}\to \MA$ is universally closed follows from \cite[Theorem 32(3)]{est1}. The fact that $\JMP^{\E,s}$ is separated follows from \cite[Proposition 26]{est1} while the fact that $\JMP^{\E,\tau}$ is proper follows from the valuative criterion for properness, which follows easily from \cite[Proposition 27 and Theorem 32(4)]{est1}.
\end{proof}

\section{Universal compactified Jacobians over $\MgP$}\label{S:Mgn}

From now on we will consider the case when $\mathfrak M_A=\MgP$ is the moduli stack of $A$-pointed stable curves of given genus $g$. See section \ref{S:pointed} for an account of the properties and notation we will be using when referring to $\MgP$.

\subsection{The stack of simple sheaves  over $\MgP$}

Let $\JgP$ denote the stack $\JMP$ in the case when $\mathfrak M=\ov{\M}_g$; we will call $\JgP$ the compactified universal Jacobian over $\MgP$.
We will write $n=|A|$ and sometimes also $\Jgn$ to denote $\JgP$ when no confusion is likely.
Sections over an $S$-scheme $T$ consist of a family of $A$-stable pointed curves $(f:X\to T; \sigma_1,\dots,\sigma_{n})$ together with a $T$-flat coherent sheaf $I$ over $X$ whose fibers over $T$ are torsion-free rank-1 simple sheaves.


Notice that $\JgP$ is certainly not of finite type: first of all it decomposes in infinitely many components
$$\JgP=\coprod_{d\in \Z} \JgP^d$$
where $\JgP^d$ denotes the open substack of $\JgP$ parametrizing families of $A$-pointed stable curves endowed with relatively flat torsion-free sheaves of relative degree $d$. For $d=0$, consider also $\JgP^{\un 0}$ to be the open substack of $\JgP^0$ whose sections consist of families of $A$-pointed stable curves as before endowed with relatively flat torsion-free sheaves of relative multidegree $\un 0$. Notice that it is natural to consider $\JgP^{\un 0}$ since it contains the universal generalized Jacobian stack $\mathcal J_{g,A}^{\un 0}$ as an open substack.

The compactified universal Jacobian stack  satisfies the following properties

\begin{prop}\label{P:UnivJac-DM}
Let $n=|A|$ and $g\geq 0$ be such that $2g-2+n>0$. Then, for any integer $d\in \mathbb Z$, the stack $\JgP^d$ is a smooth and irreducible Deligne-Mumford stack of dimension $4g-3+n$ endowed with a representable morphism $\pi$ onto $\MgP$ satisfying the existence part of the valuative criterion for properness.
\end{prop}

\begin{proof}
The fact that $\JgP$ is a Deligne-Mumford algebraic stack and that $\pi$ is representable and satisfies the existence part of the valuative criterion for properness follows from Proposition \ref{P:algebraicity}, and the same clearly holds for $\JgP^d$. Finally, $\JgP$ is smooth and irreducible since it is an open substack of $\JMP^{lp}$, which, as we have seen in Proposition \ref{P:loc-plan}, is smooth and irreducible.
\end{proof}

\subsubsection{Universal family and clutchings}

We will now see that the operations we defined in \ref{S:UFPC} for the moduli stack of stable pointed curves, namely the universal family and the clutchings can also be defined on $\JgP$ or in some of its components.

Denote by $\UgP$ the universal family over $\JgP$. Sections of $\UgP$ over an $S$-scheme $T$ consist of a section of $\JgP$ over $T$, say $(f:X\to T; \sigma_1,\dots,\sigma_{|A|};I)$ together with an extra section $\sigma:T\to X$ with no stability requirements. The universal family $\UgP$ is endowed with a universal sheaf $\I$ defined by assigning to a family  $(f:X\to T; \sigma_1,\dots,\sigma_{|A|};I;\sigma)$ the sheaf $\sigma^*(I)$ on $T$.

The natural forgetful morphism $\pi:\UgP\to \JgP$ is naturally endowed with $|A|$ sections, that we will denote also by $\sigma_i$, for $i\in A$. It is possible to define natural morphisms
$$\lambda: \UgP\to\JgPx$$
by stabilizing the family as in (\ref{F:univFamily}) and by pulling back the sheaf (or some twisted version of it) to the stabilized family. However, notice that $\lambda$ will not be an isomorphism this time: we can not recover all torsion-free simple rank $1$ sheaves on a curve from a curve with fewer components without requiring conditions on the multidegrees of the sheaves.

By abuse of notation, denote again by $\Di$ and $\Da$ the pullback of the boundary divisors $\Di$ and $\Da$ from $\MgP$, which are again divisors on $\JgP$.

Let $(f:X\to T;\sigma_1,\dots,\sigma_{|A|},\sigma_x,\sigma_y; I)\in \ov{\mathcal J}_{g-1,A\cup\{x,y\}}$ and $(\ov f:\ov X\to T;\ov \sigma_1,\dots,\ov \sigma_{|A|})=: \ov\xi_{irr}(f:X\to T;\sigma_1,\dots,\sigma_{|A|},\sigma_x,\sigma_y)\in \JgP$. Then
 $(\xi_{irr}^f)_*(I)$ is a simple torsion free sheaf on $\ov X$. In fact, an endomorphism of  $(\xi_{irr}^f)_*(I)$ would induce an endomorphism of  $I$ as well. Notice that $(\xi_{irr}^f)_*(I)$ will be non locally-free in the new nodes of the fibers of $\ov f$. We conclude  that the morphism $\xi_{irr}$ defined in \ref{E:chi-irr} induces also a morphism $\ov\xi_{irr}$ from $\ov{\mathcal J}_{g-1,A\cup\{x,y\}}$ to $\JgP$.

Let us now consider the morphism $\xi_{b,B}$ defined in $\ref{E:chi-bB}$ and let us see that it extends to a morphism $\ov\xi_{b,B}$ from $\ov{\mathcal J}_{b,B\cup\{x\}}\times \ov{\mathcal J}_{g-b,B^c\cup\{y\}}$. Consider then $(f_1:X_1\to T;\sigma_1,\dots,\sigma_{|B|},\sigma_x; I_1)\in \ov{\mathcal J}_{b,B\cup\{x\}}$ and $(f_2:X_2\to T;\sigma_1',\dots,\sigma_{|B^c|}',\sigma_y; I_2)\in \ov{\mathcal J}_{g-b,B^c\cup\{y\}}$. It follows from Example 37 in \cite{est1} that $\xi_{b,B}^{f_1,f_2}$ induces an isomorphism between pairs $(I_1,I_2)$ of simple torsion-free sheaves on $(X_1,X_2)$ and simple torsion-free sheaves $I$ on $\ov X:=\xi_{b,B}^{f_1,f_2}(X_1,X_2)$ given fiberwise by gluing the sheaves on the new (separating) node (notice that a simple sheaf on $\ov X$ must be locally free on the new node).

Summing up, we have

\begin{prop}\label{P:clutch-forget}
There are clutching morphisms
$$\ov\xi_{irr}:\ov{\mathcal J}_{g-1,A\cup\{x,y\}}\rightarrow \JgP$$
and
$$\ov\xi_{b,B}:\ov{\mathcal J}_{b,B\cup\{x\}}\times \ov{\mathcal J}_{g-b,B^c\cup\{y\}}\rightarrow \JgP$$
defined by applying $\ov\xi_{irr}$ and $\ov\xi_{b,B}$, respectively, to the families of $A$-pointed curves.
The images of $\ov\xi_{irr}$ and $\ov\xi_{b,B}$ naturally lie, respectively, in $\Di$ and $\Da$.

\end{prop}

\subsection{Fine compactified universal Jacobian stacks}

\subsubsection{Polarisations on $\MgP$}

Following section \ref{S:finequotients}, recall that a $d$-polarisation (or just polarisation) on $\MgP$ is a vector bundle $\E$ on $\CgP$ of rank $r>0$ and degree $r(d-g+1)$ for some integer $d$.
This prescribes, for each family $(f:X\to T; \sigma_1,\dots, \sigma_{|A|})\in \MgP(T)$, a relative polarisation of $X$ over $T$ in a functorial way.


In order to define natural and explicit polarisations on $\MgP$ we will use below the sheaf of regular functions of $\CgP$, the dualizing sheaf of $\pi:\CgP\to\MgP$, the sections of $\pi$ and the boundary divisors.



\subsubsection{Substacks of universal Jacobians of (semi)stable sheaves}

\begin{defi}\label{D:polarisation2}
Let $\E$ be a $d$-polarisation on $\MgP$ and $\sigma$ a section of $\CgP\to \MgP$. Accordingly to Definition \ref{D:polarisation}, we will denote with $\JgP^{\E,ss}$ (resp. $\JgP^{\E,s}$, $\JgP^{\E,\sigma}$) the substack of $\JgP$ parametrizing families of $A$-pointed stable curves together with $\E$-semistable (resp. $\E$-stable, $(\E,\sigma)$-quasistable) torsion-free simple sheaves.
\end{defi}

Proposition \ref{P:JacProp} applied to $\MgP$ yields the following.

\begin{prop}\label{P:univJacProp}
Let $\E$ be a $d$-polarisation on $\MgP$
and $\sigma$ a section of $\CgP\to \MgP$.
Then the stacks $\JgP^{\E,ss}$ (resp. $\JgP^{\E,s}$, $\JgP^{\E,\sigma}$) are smooth and irreducible Deligne-Mumford stacks of dimension $4g-3+n$ endowed with a representable
morphism $\pi$ onto $\MgP$. In fact, they are open substacks of $\JgP$ of finite type and
\begin{enumerate}[(i)]
\item $\JgP^{\E,ss}$ is universally closed;
\item $\JgP^{\E,s}$ is separated;
\item $\JgP^{\E,\sigma}$ is proper.
\end{enumerate}
\end{prop}
\begin{proof}
The proof is an immediate consequence of Propositions \ref{P:JacProp}, \ref{P:UnivJac-DM} and of the fact that $\MgP$ is a proper stack.


\end{proof}

\begin{rem}\label{quotientstack}
The fact that  $\JgP^{\E,ss}$, $\JgP^{\E,s}$, $\JgP^{\E,\sigma}$ are Deligne-Mumford stacks endowed with a representable morphism to $\MgP$, which is a quotient stack, implies that themselves are quotient stacks (see e.g. \cite[section 5]{kreschGDM}).
\end{rem}


\subsection{Coarse moduli spaces}\label{S:coarse}

Let $\E$ be a $d$-polarisation on $\MgP$ and $\sigma$ a section $\sigma:\MgP\to \ov{\mathcal M}_{g,A\cup\{x\}}$.

The aim of this section is to show that $\JgP^{\E,\sigma}$ has a projective moduli space.

First of all, by directly applying Keel-Mori's criterion, we have the following.

\begin{prop}
Let $\E$ be a $d$-polarisation on $\MgP$ and $\sigma:\MgP\to \ov{\mathcal M}_{g,A\cup\{x\}}$ a section of the universal family morphism of $\MgP$. Then $\JgP^{\E,\sigma}$ admits a coarse moduli space $\JgA^{\E,\sigma}$, which is a proper algebraic space.
\end{prop}

\begin{proof}
The fact that $\JgP^{\E,\sigma}$ admits a separated coarse moduli space is a direct consequence of Proposition \ref{P:univJacProp} using Keel-Mori's result (see \cite[Corollary 1.3]{KM}), which asserts that (separated) Deligne-Mumford stacks admit  (separated) algebraic spaces which are coarse moduli spaces for the stack. Moreover, properness follows from the valuative criterion for properness, which follows easily from \cite[Proposition 27 and Theorem 32(4)]{est1}.
\end{proof}

Let then  $\JgA^{\E,\sigma}$ be the coarse moduli space for $\JgP^{\E,\sigma}$, which is proper by what we just said. In order to prove that $\JgA^{\E,\sigma}$ is projective we will use Koll\'ar's semipositivity criterion (see \cite{kollar}).
Let us start by recalling the definition of semipositive vector bundle.

\begin{defi}\cite[Definition 3.3]{kollar}
Let $V$ be a locally free sheaf on a scheme (or algebraic space) X. Then $V$ is {\em semipositive} if for all proper maps $f:C\to X$ from a curve $C$, any quotient bundle of $f^*V$
has non-negative degree.
\end{defi}

The second ingredient is that of a finite classifying map. Let $W$ be a vector bundle on $X$ of rank $w$ with structure group $\rho:G\to GL_w$.
 Given a quotient bundle $Q$ of $W$ of rank $q$, the {\em classifying map}  is the set theoretical map on the closed points of $X$, defined as
$$u:X\to Gr(w,q)/G$$
where $Gr(w,q)$ denotes the Grassmannian of $q$-dimensional quotients of a fixed $w$-dimen\-sio\-nal space.
Then we say that the classifying map $u$ is finite if its fibers are finite and all points of the image have finite stabilizer.

The following is Koll\'ar's Ampleness Lemma.

\begin{fact}{\cite[Lemma 3.9]{kollar}}\label{F:ampleness}
Let $X$ be a proper algebraic space, $W$ a semipositive vector bundle on $X$ of rank $w$ and with structure group $G$ and let $Q$ be a quotient vector bundle of $W$ with rank $q$. Assume that $W$ satisfies one of the several technical conditions listed in \cite[Sec. 3]{kollar} and that the classifying map is finite. Then $\det Q$ is ample and, in particular, $X$ is projective.
\end{fact}

We will now apply Koll\'ar's Ampleness Lemma to show that there is an ample line bundle on $\JgA^{\E,\sigma}$.

\begin{prop} \label{P:positivity}
Let $\mathcal L$ be a very ample line bundle on $\MgPx$, $\E$ a $d$-polarisation and let $(f:X\to T,\sigma_i:T\to X,F)$ a section of $\JgP^{\E,\sigma}$ over $T$. Denote by $ L_f$ the line bundle that $\mathcal L$ induces on $X$. Then, there is an integer $k$ such that
\begin{enumerate}
\item $R^1(f_*(F\otimes  L_f^j))=0$ for $j\geq k-1$;
\item $f^*f_*( F\otimes{ L}_f^j)\to F\otimes   L_f^j$ is surjective for $j\geq k$.
\end{enumerate}
Moreover, for $K:=Ker[f^*(f_*(F\otimes  L_f^k))\to F\otimes  L_f^k]$, there is an integer $m$ such that $f^*(f_*(K\otimes  L_f^j))\to K\otimes  L_f^j$ is surjective for $j\geq m$.
\end{prop}

\begin{proof}

From \cite[Cor. 1.5]{knudsen}, $(1)$ and $(2)$ follow if we prove that, for each section $(X, p_i, F)$ of $\JgP^{\E,\sigma}$ over an algebraically closed field $k$, we have that
\begin{itemize}
\item $H^1(X,F\otimes L_X^j)=0$,
\item $F\otimes L_X^j$ is generated by its global sections,
\end{itemize}
where $L_X$ denotes the line bundle induced by $\mathcal L$ on $X\rightarrow k$.
Now, by Serre duality, we have that
$$H^1(X,F\otimes L_X^j)=H^0(X, F^*\otimes L^{-j}),$$
so we must show that there is an integer $m$ such that $H^0(X, F^*\otimes L^{-j})=0$ for all $j\geq m$ and for all sections $(X,p_i,F)$ of $\JgP^{\E,\sigma}$ over an algebraically closed field $k$.

For all $g$ and $n=|A|$, there is a finite number of curves in $\MgP$, so by \ref{E:stabCond}, the stability conditions imposed by the polarisation $\E$ give a finite number of inequalities of the form
$\deg_ZF\geq n_{Z}$, where $n_{Z}$ is a number depending only on $Z\subset X$.
This implies that, for $j$ big enough, we can control the degree of $F^*\otimes L^{-j}$ to be negative on all subcurves $Z$ of $X$ (we must use here again the fact that there are a finite number of topological types of curves in $\MgP$). Take $m$ such that the degree of $F^*\otimes L^{-j}$ on each subcurve $Z$ of a curve $X$ in $\overline M_{g,A}$ is at most equal to $-2$ for all $j\geq k$.
Then we have that $h^0(X, F^*\otimes L^{-j})=0=h^1(X,F\otimes L_X^j)$ for all curve $X$ in $\overline M_{g,A}$ and, moreover, that $F\otimes L_X^j$ is generated by its global sections. In fact, by Riemann Roch, for any such geometric point $p\in X$,
$h^0(X, F\otimes L^{j}(-\mathfrak m_p))<h^0(X, F\otimes L^{j})$, where $\mathfrak m_p$ denotes the ideal sheaf of $p$ in $X$. This is enough to get $(1)$ and $(2)$.

Now, the second part follows easily using again the boundedness of $\MgP$ and we conclude.

\end{proof}

From Proposition \ref{P:positivity} above and Koll\'ar's results in \cite[Sec. 6]{kollar} we get the following result.

\begin{cor} \label{C:positivity}
Notation as above. Then the bundle
 $W:=(f_*(F\otimes L^k)\otimes f_*(L^m))^*$ is semipositive. Moreover, the dual of the kernel of the multiplication map
 $$\mathcal Q_f:= {\rm ker}[f_*(F\otimes L^k)\otimes f_*(L^m)\to f_*(F\otimes L^{k+m})]^*$$ is a quotient of $W$ such that $\det \mathcal Q_f$ is relatively ample over $X\to S$.
 \end{cor}

 \begin{proof}
 The fact that $(f_*(F\otimes L^k)\otimes f_*(L^m))^*$ is semipositive and that it satisfies the technical conditions on Fact \ref{F:ampleness} follows from Proposition \ref{P:positivity} above by applying the same reasoning as in \cite[Prop. 6.8]{kollar} and in \cite[Cor. 6.9]{kollar}, respectively.
 The fact that the classifying map associated to the quotient $W\twoheadrightarrow Q_f$ is finite follows by applying Proposition \ref{P:positivity} in the proof of Theorem 6.4 in loc. cit..The result is now a consequence of Fact \ref{F:ampleness}.
 \end{proof}

 It is clear that the line bundle $\det \mathcal Q_f$ defined on Corollary \ref{C:positivity} for each section $(f:X\to T,\sigma_i:T\to X,F)$ of $\JgP^{\E,\sigma}$ over $T$ glues to give a relatively ample line bundle $\det \mathcal Q$ on $\JgP^{\E,\sigma}\to \MgP$. This line bundle won't necessarily descend to $\JgA^{\E,\sigma}$, but a power of it will. We thus conclude.

 \begin{cor}
 Notation as above. Then $\JgA^{\E,\sigma}$ is a projective coarse moduli space for $\JgP^{\E,\sigma}$.
 \end{cor}


\subsection{Explicit polarisations on $\MgP$}\label{S:expol}
According to Definition \ref{D:polarisation} above, in order to define a $d$-polarisation on $\MgP$, we must give a vector bundle $\E$ on $\CgP$ of a certain rank $r>0$ and degree $r(d-g+1)$. Actually, in view of Proposition \ref{P:tensor} and for our convenience, we will give in general a $\mathbb Q$-vector bundle instead, where with $\mathbb Q$-vector bundle we simply mean a direct sum of $\mathbb Q$-line bundles here.

Let then $s; a_i, i\in A $ and $\alpha_{(b,B)}$ as $(b,B)$ ranges through the admissible types of nodes in the boundary of $\MgP$ be rational numbers such that
$$\frac{s(2g-2)+\sum_{i\in A} a_i}{r}\in \Z.$$
Let $\pi:\CgP\to\MgP$ be the natural forgetful functor and consider the $\mathbb Q$-vector bundle  in $\CgP$ defined by using the relative dualizing sheaf $\omega={\omega}_{\pi}$, the sections $\sigma_i$ of $\pi$ and the boundary divisors $\da$ as follows:
\begin{equation}\label{D:expol}
\E=\left (\omega^s(\sum_{i\in A}a_i\sigma_i )\otimes \sum_{(b,B)}\alpha_{(b,B)}\da\right )\oplus \mathcal O_{\CgP}^{r-1}.
\end{equation}
Then $\rk \E=r$ and $\deg \E=s(2g-2)+\sum_{i\in A}a_i$, which implies that
$$d=\frac{s(2g-2)+\sum_{i\in A} a_i}{r}+g-1$$
which, by hypothesis, is an integer number.
Notice also that, according to Theorem \ref{T:arbcorn} above, considering vector bundles of the form introduced in \eqref{D:expol} is not really restrictive.


Fix a polarisation $\E$ on $\MgP$ and a section $\sigma$ of $\pi$ as above. Given a section of $\MgP$ over an algebraically closed field $k$, i.e., an $A$-pointed (stable) $k$-curve $(X;p_1,\dots,p_{|A|})$, let $E$ be the polarisation on $X$ induced by $\E$; i.e. $E:=\E_{|X}$.
Write $E$ as
$$E=\left (\omega_X^s(\sum_{i\in A}a_ip_i)\otimes \sum_{(b,B)}\alpha_{(b,B)}\daX\right )\oplus \mathcal O_X^{r-1}$$
and denote by $p\in X$ the image of the section $\sigma$ on $X$.
Let us describe the stability condition induced by $E$ (and $(E,p)$) on sheaves over $X$.


Given $Y\subseteq X$, set
$$E_Y:=E_{|Y}=\left({\omega_{X}^s}_{|Y}(\sum_{p_i\in Y}a_ip_i)\otimes \sum_{(b,B)}\alpha_{(b,B)}\daX_{|Y}\right)\oplus \mathcal O_Y^{r-1}.$$
Then, we have that
$$q_Y^E=\frac{\deg E_Y}{r}+\frac{w_Y}{2}=\frac{sw_Y+\sum_{p_i\in Y}a_i+\sum_{(b,B)}\alpha_{(b,B)}\daY}{r}+\frac{w_Y}{2}$$
where $w_Y:=\deg_Y\omega_X=2g_Y-2+k_Y$ and $\daY:=\deg_Y(\daX)$.

Then Definition \ref{semiquasistabledef} in this case can be rephrased by saying that a torsion-free sheaf $I$ on $X$ is $E$-(semi)stable if $\deg I=\frac{s(2g-2)+\sum_{i\in A} a_i}{r}+g-1$ and if given $\emptyset \neq Y\subsetneq X$,
\begin{equation}
\label{semistabledefdegree}\deg I_Y(\geq)> q^E_Y-\frac{k_Y}{2}.
\end{equation}
Similarly, $I$ is said to be $(E,p)$-quasistable if it is $E$-semistable and if inequality (\ref{semistabledefdegree}) above is strict whenever $p\in Y$.

\begin{rem}\label{R:stab-integer}
Let $\E=\left (\omega^s(\sum_{p\in P}a_p\sigma_p)\otimes \sum_{(b,B)}\alpha_{(b,B)}\da\right )\oplus \mathcal O_{\CgP}^{r-1}$ be a $d$-polarisation on $\MgP$. Then, for any integer $m\in \Z$, one checks immediately that the $\mathbb Q$-vector bundle
$\E_m=\left (\omega^{ms}(m\sum_{i\in A}a_i\sigma_i)\otimes m\sum_{(b,B)}\alpha_{(b,B)}\da\right )\oplus \mathcal O_{\CgP}^{mr-1}$
is again a $d$-polarisation on $\MgP$ and defines exactly the same stability conditions. We can thus assume that all the constants $s$, $a_i$ and $\alpha_{(b,B)}$ are integers and divisible by a certain integer.
\end{rem}

While the degrees of the relative dualizing sheaves and of the punctures on subcurves $Y$ of $X$ are clearly computed as we set above, for computing $\daY$ we will need a little argument, that we will describe in \ref{S:boundary}.

\subsubsection{Boundary divisors}\label{S:boundary}
Let $1\leq b\leq g$ and $B\subseteq A$ be such that if $b=0$ (resp $b=g$) then $|B|\geq 2$ (resp. $|B^c|\geq 2$).
Let $(X;p_i)$ be an $A$-pointed (stable) curve of genus $g$ and let $p$ be a node of $X$ of type $(b,B)$. The fiber of $\pi_x:\MgPx\to \MgP$ over $[(X;p_i)]$ can be seen as a curve over $X$ (obtained by varying the extra section $x$), which we will denote by $[(X;p_i)]_{x\in X}$. Let $C$ be an irreducible subcurve of $X$ and consider the restriction of  $[(X;p_i)]_{x\in X}$ to $C$: $[(X;p_i)]_{x\in C}$.
In order to compute the degree of $\da$ in $C$ recall the following relation of divisor classes on $\MgP$:
\begin{equation}\label{E:pull-bound}
\pi_x^*(\da)=\da + \delta_{b,B\cup\{x\}}
\end{equation}
Start by observing  that if $C$ has no nodes of type $(b,B)$, then the degree of $\da$ on $[(X;p_i)]_{x\in C}$ is equal to zero.

Suppose now that $C$ has a node of type $(b,B)$; recall that it means that the normalization of $X$ at $p$ is given by two disconnected pointed curves $X_p$ and $X_p'$ of genera $b$ and $g-b$ and markings $B$ and $B^c$, respectively.
Then it is easy to see that
\begin{equation}\label{E:deg-comp}
\deg_{C}(\da)=
\begin{cases}
1 & \text{if $p\in C$ and $C\subseteq X_p$,}
\\
-1 & \text{if $p\in C$ and $C\subseteq X_p'$.}
\end{cases}
\end{equation}
In fact, if $C\subseteq X_p$, then $[(X;p_i)]_{x\in C}$ is a curve contained in $\mathcal D_{b,B\cup\{x\}}$ and meets $\Da$ in the point $x=p$.  Arguing as in \cite[Lemma1]{arbcorn}, we conclude that $\deg_{[(X;p_i)]_{x\in C}}\da=1$. In the case when $C\subseteq X_p'$, then  $[(X;p_i)]_{x\in C}$ is a curve contained in $\Da$, so again arguing as in \cite[Lemma1]{arbcorn} we conclude that $\deg_{[(X;p_i)]_{x\in C}}\da=-1$. In fact, simply by base-change, the degree on $[(X;p_i)]_{x\in C}$ of the pullback of $\da$ via $\pi_x$, $\pi_x^*(\da)$, must be equal to zero since the image of the curve $[(X;p_i)]_{x\in C}$  on $\MgP$ is just a point on $\MgP$. Therefore, we get from \eqref{E:pull-bound} that
$$\deg_{[(X;p_i)]_{x\in C}}\da=- \deg_{[(X;p_i)]_{x\in C}}\delta_{b, B\cup\{x\}}.$$
We conclude with the following

\begin{prop}
Let $(X;p_i)$ be an $A$-pointed (stable) curve of genus $g$ and let $Z\subseteq X$ be a subcurve of $X$. Then
$$\deg_Z\da=\sharp\{p\in Z \cap N_{b,B}(X) \mbox{ and } Z\subseteq X_p\}-\sharp\{p\in Z\cap N_{b,B}(X) \mbox{ and } Z\subseteq X_p'\},$$ where $N_{b,B}(X)$ denotes the set of nodes of type $(b,B)$ of $X$.
\end{prop}

\begin{proof}
The result is an easy adaptation of formula \eqref{E:deg-comp} above to the general case.
\end{proof}

\subsubsection{Canonical polarisations on $\MgP$}\label{S:canonical}

Given an integer $d$ and a tuple of numbers $\un a=(a_1,\dots,a_{|A|})$, consider the canonical polarisation associated to $(d,\un a)$ to be
$$\E^{d,\un a}:=\left( \omega\left(\sum_{i\in A}a_i\sigma_i\right)\right)^{d-g+1}\oplus \O^{\oplus( 2g-3+\sum_{i\in A}a_i)}.$$
Then
$$\frac{\deg \E^{d,\un a}}{\rk \E^{d,\un a}}=\frac{(d-g+1)(2g-2+\sum_{i\in A}a_i)}{2g-2+\sum_{i\in A}a_i}=d-g+1,$$
and $\E^{d,\un a}$ is a $d$-polarisation on $\MgP$.
So $\E^{d,\un a}$-(semi)stable torsion free sheaves must have degree $d$ and, for a section $(X;p_i\in A;I)$ of $\JgP$ over an algebraically closed field $k$, inequality (\ref{semistabledefdegree}) yields that for a subcurve $Y\subseteq X$,
$$\deg I_Y\geq \frac{(d-g+1)(\omega_Y+\sum_{p_i\in Y}a_i)}{2g-2+\sum_{i\in A}a_i}+\frac{w_Y}2-\frac{k_Y}2,$$
which is easily seen to be equivalent to say that
\begin{equation}\label{canonical}
\deg I_Y+\sum_{p_i\in Y}\frac{a_i}2\geq \left(d+\sum_{i\in A}\frac{a_i}2\right)\frac{\omega_Y+\sum_{p_i\in Y}a_i}{2g-2+\sum_{i\in A}a_i}-\frac{k_Y}2.
\end{equation}
Notice that inequality (\ref{canonical}) coincides with inequality (1.6) in \cite{li} which describes GIT- semistability of nodal weighted curves with weight $\un a$ embedded in a suitable Chow scheme. For $A=\emptyset$, inequality (\ref{canonical}) coincides with the so-called ``basic inequality" discussed in \cite{cap} to describe GIT-semistability of nodal curves embedded in a suitable Hilbert scheme (see also \cite{CP} for a discussion on canonical polarisations in the non-pointed case).

\subsubsection{Varying the polarisations}

In this section we will play a little bit with the polarisations we introduced in \ref{S:expol} and explore some properties of
the compactified universal Jacobians that they yield.

The first question one may ask is: do different polarisations obtained by varying the coefficients in \eqref{D:expol} give non isomorphic moduli spaces?
We will observe that, even in the case when $A\neq \emptyset$, the answer is yes.

Start by considering the case when $A=\emptyset$.
Notice that all $d$-polarisations of the type $\E=\omega^s\oplus \mathcal O^{r-1}$ in $\Mg$ yield the canonical compactification. In fact, since $\frac s r=d-g+1$, we always get the same inequalities in \eqref{semistabledefdegree}. So, in this case, only by considering polarisations $\E$ with some of the $\alpha_{b}$'s different from $0$, we can get different compactifications.

Take $d=g-1$ and a curve $X=X_1\cup X_2$ of compact type in $D_{g_1}$, where $g_1=g_{X_1}$. Then, given a torsion-free sheaf $I$ on $X$, the inequalities \eqref{semistabledefdegree} induced by the canonical polarisation $\E^d:=\E^{d,\underline 0}$ on $I$ are
$$\deg I_{X_1}\geq q_{X_1}^{E^d}-\frac{k_{X_1}}{2}=g_{X_1}-1,$$
$$\deg I_{X_2}\geq q_{X_2}^{E^d}-\frac{k_{X_2}}{2}=g_{X_2}-1.$$
So, there are two possibilities for the multidegree of $I$ on $X$: $(g_{X_1}-1, g_{X_2})$ and $(g_{X_1}, g_{X_2}-1)$. Since $I$ has to be simple, we conclude that it has to be a line bundle and the fiber of $\overline{\mathcal J}_{g}^{\E^d,ss}$ over $[X]$ has two components, each one parametrizing sheaves of one of the allowed multidegrees.

Consider now $\E=\left (\omega^s\otimes\delta_{g_1}\right )\oplus \mathcal O^{r-1}$. Then an $\E$-semistable torsion-free sheaf on $X$ has to be a line bundle and its multidegree on $X$ must satisfy
$$\deg I_{X_1}\geq q_{X_1}^{E}-\frac{k_{X_1}}{2}=g_{X_1}-1+\frac{1}{2g-2},$$
$$\deg I_{X_2}\geq q_{X_2}^{E}-\frac{k_{X_2}}{2}=g_{X_2}-1-\frac{1}{2g-2}.$$
So, since $g\geq 2$, there is only one possibility for the multidegree of $I$ on $X$:  $(g_{X_1}, g_{X_2}-1)$.  We conclude that the fiber of $\overline{\mathcal J}_{g}^{\E,ss}$ over $[X]$ is irreducible, so $\overline{\mathcal J}_{g}^{\E^d,ss}\ncong \overline{\mathcal J}_{g}^{\E,ss}$.

By proceeding in the same way for all boundary divisors $D_b$, for $b>0$ and for all degrees $d$ we easily conclude the following

\begin{prop}
For any integer $d$, there is a $d$-polarisation $\E$ on $\overline{\mathcal M}_g$ which is is general away from the
locus of curves of two components meeting in more than one node.
\end{prop}



Consider now the situation when $A\neq \emptyset$. Then we can still consider polarisations of the type $\E^d=\omega^s\oplus \mathcal O^{r-1}$ in $\MgP$ and get, as before, that the fiber of $\overline{\mathcal J}_{g,A}^{\E^d,ss}$ over curves $[X]$ as above may have two components, each one representing sheaves of one of the allowed multidegrees. Again, by considering a polarisation $\E$ modified by adding a boundary divisor as above we get that the fiber of $\overline{\mathcal J}_{g,A}^{\E,ss}$ over $[X]$ is irreducible as in the previous situation. This shows that the compactifications $\JgP^{\E,ss}$ do depend on $\E$, also in the case when $A\neq \emptyset$.
Notice that, however, the situation is different in this case since we can also play with the marked points and, by fixing a section $\sigma$ of $\pi:\CgP\to\MgP$, consider the fine compactifications $\JgP^{\E,\sigma}$ instead. In this case, the number of irreducible components of the different compactifications over each curve are always the same (see \cite[Theorem 3.1]{MV}). However, this fact is not enough to conclude that the different compactifications are equal, so the following question remains:
\begin{question}
Let $\E$ be a $d$-polarisation of $\MgP$ and $\sigma$ a section of $\pi:\CgP\to\MgP$. Do the compactifications $\JgP^{\E,\sigma}$ depend on $\E$ and $\sigma$?
\end{question}

\begin{rem}
The answer to the above question posed on fine compactifications over single curves rather than the universal compactifications is negative as proved in \cite[Theorem B]{MRV1}.
\end{rem}

As we have already observed in the introduction, Kass and Pagani have recently constructed in \cite{KP} a compactification of the universal Jacobian stack of degree $d=g-1$ over $\MgP$ in the case when $A\neq \emptyset$.
The construction in loc. cit. depends upon the choice of a stability parameter $\phi$ and yields a Deligne-Mumford stack $\ov{\mathcal J}_{g,A}(\phi)$, which is representable over $\MgP$. In fact it is shown in loc. cit. that the stability condition
given by $\phi$ can be non-canonically identified with Simpson's slope stability with respect to a suitable polarisation, so it suffices to use Simpson's representation result in \cite[Theorem 1.21]{simpson} to prove existence.

The stability parameter $\phi$ consists of assigning real numbers $\phi(v_i)$ to the vertices $v_i$ of the dual graphs of each stable marked curve in such a way that this assignment is compatible with the contraction of edges in the graphs (see Definitions 3 and 5 in \cite{KP}).
The moduli stack $\JgP(\phi)$ parametrizes, as in our case, torsion-free rank-1 sheaves which are $\phi$-semistable, where the condition of $\phi$-semistability coincides with ours  by replacing the numbers $q^E$ by the $\phi$'s (see (10) in loc. cit.).
Such a stability parameter is said to be non-degenerate if all $\phi$-semistable sheaves are necessarily stable.
Over the locus $\MgP^{(0)}$ parametrizing tree-like curves, the stability parameter $\phi$ is described explicitly. In particular, it follows from Lemma 2 in loc. cit. that it suffices to prescribe the values of $\phi$ over the locus of curves with two smooth components meeting at one (separating) node.

In what follows, we will show that, given a stability condition $\phi$ over $\MgP^{(0)}$,  it is always possible to choose the coefficients $s, a_i$ and $\alpha_{(b,B)}$ to form a $g-1$-polarisation $\E$  as in \eqref{D:expol} such that $\JgP^{E,ss}$ is isomorphic to $\ov{\mathcal J}_{g,A}(\phi)$ over that locus.

\begin{prop}\label{P:compare}
Let $\phi$ be a non-degenerate stability parameter over $\MgP^{(0)}$ as defined above.
Then there is a $(g-1)$-polarisation $\E$ as in \eqref{D:expol}  such that $\JgP^{E,ss}$ is isomorphic to $\ov{\mathcal J}_{g,A}(\phi)$ over $\MgP^{(0)}$.
\end{prop}

\begin{proof}
According to Lemma 2 in \cite{KP}, the stability parameter $\phi$ is determined by assigning real numbers $(\phi(v_{b,B}), \phi(v_{g-b, B^c}))$ to the vertices $v_{b,B}$ and $v_{g-b,B^c}$ of the dual graph of stable marked  curves with two components $X_{b,B}$ and $X_{g-b,B^c}$ meeting at one node of type $(b,B)$. So, in order to conclude, it suffices to find a $(g-1)$-polarisation on $\MgP$ such that, for all admissible type of node $(b,B)$ and given a curve $X=X_{b,B}\cup X_{g-b,B^c}$ as above, we have that $q_{X_{b,B}^E}=\phi(v_{b,B})$ (this immediately implies that also $q_{X_{g-b,B^c}^E}=\phi(v_{g-b,B^c})$). In fact, since the stability parameter is non-degenerate, all $\phi$-semistable sheaves are necessarily simple, so we conclude by comparing the stability condition yield by $\E$ and by $\phi$ on the aforementioned curves.

We start by observing that it suffices to consider the case when  $(\phi(v_{b,B}), \phi(v_{g-b, B^c}))$ are rational. In fact, the space of stability parameters is cut out by stability hyperplanes defining chambers along which all the stability parameters yield the same compactification. As the stability hyperplanes are defined by rational equations, there are rational stability parameters in all chambers, so we can pick up one of those.

Consider on $\MgP$ a $(g-1)$-polarisation of the type
$$\E=\mathcal O_{\CgP}\left ( \sum_{(b,B)}\alpha_{(b,B)}\da\right ).$$
In fact, $\E$ yields a $(g-1)$-polarisation as $\frac{\deg \E}{\rk \E}=0$.
For each admissible type of node $(b,B)$, set
$$\alpha_{(b,B)}:=
\begin{cases}\phi(v_{b,B})-b+\frac 12 &\text{ if } 1\in B\\
0 &\text{ otherwise.}
\end{cases}
$$
Then, given a two-component curve $X=X_{b,B}\cup X_{g-b, B^c}$ as above, and assuming that $1\in B$, we have that
$$q_{X_{b,B}}=\alpha_{(b,B)}+b-\frac 12=\phi(v_{b,B})-b+\frac 12+b-\frac 12=\phi(v_{b,B}).$$
This suffices to conclude that
$\JgP^{\E,ss}\cong\JgP(\phi)$ over $\MgP^{(0)}$.
\end{proof}




\subsection{Forgetful morphisms}

Consider the forgetful morphism
\begin{equation*}
\begin{array}{rccc}
\pi_x:&\MgPx&\longrightarrow &\MgP\\
&(\xymatrix{
X\ar[r]_{f} & T \ar @/_/[l]_{\sigma_i,\sigma_x}
})&\longmapsto & (\xymatrix{
\bar X\ar[r]_{\bar f} & T \ar @/_/[l]_{\bar\sigma_i}
})
\end{array}
\end{equation*}
and a $d$-polarisation $\E_x$ on $\MgPx$. The aim of this section is to show that, if $\E_x$ satisfies certain properties and if $\sigma$ is a section of $\pi':\CgPx\to\ov{\mathcal M}_{g,A\cup\{x\}}$ inducing a section $\sigma$ of $\pi:\CgP\to\MgP$, then $\pi_x$ induces a morphism
\begin{equation}\label{E:forg}
\begin{array}{rccc}
\pi_x:&\JgPx^{\E_x,\star}&\longrightarrow &\JgP^{\E,\star}\\
&(\xymatrix{
X\ar[r]_{f} & T \ar @/_/[l]_{\sigma_i,\sigma_x}
}, I)&\longmapsto & (\xymatrix{
\bar X\ar[r]_{\bar f} & T \ar @/_/[l]_{\bar\sigma_i}
}, {\pi_x}_*(I))
\end{array}
\end{equation}
where $\star=ss, s$ or $\sigma$ and $\E$ is a certain $d$-polarisation on $\MgP$ induced by $\E_x$.

Let $(X; p_1,\dots, p_{|A|},p_x;I)$ be a section of $\JgPx$ over an algebraically closed field $k$ and let $Y\subseteq X$ be an irreducible component of $X$ contracted by $\pi_x$. We say that $\E_x$ satisfies the condition $(*)$ if for all such $Y$ we have that
$$q_Y^{E_x}=0,\hspace{2cm} (*)$$
where $E_x$ denotes the polarisation induced by $\E_x$ on $X$. In other words, $q_Y^{E_x}$ must be equal to $0$ for all irreducible components $Y$ isomorphic to $\mathbb P^1$ such that:
\begin{enumerate}[(a)]
\item $k_Y=2$, $p_x\in Y$ and, $\forall i\in A, p_i\notin Y$;
\item $k_Y=1$, $\exists i $ such that $p_i,p_x\in Y$ and, $\forall i\neq j\in A, p_j\notin Y$.
\end{enumerate}
In particular, $q_Y^{E_x}=0$ for $Y$ as in $(a)$ implies that $a_x=0$.

\begin{defi}
Let $Z\subsetneq X$ be a proper subcurve of a nodal curve $X$ and let $I$ be a torsion-free rank-1 simple sheaf on $X$. Set $$NF(Z):=\{p\in Z\cap \overline{Z^c}:  \mbox{ is not locally free at }p\}$$ and define
$$d(Z):=\deg I_Z+\sharp NF(Z).$$
\end{defi}

\begin{lem}\label{L:deg-NF}
Let $\E_x$ be a $d$-polarisation on $\MgPx$ satisfying condition $(*)$. Then, given an irreducible component $Y\subsetneq X$ contracted by $\pi_x$, we have that
\begin{enumerate}[(i)]
\item $d(Y)=\deg I_Y=0$ if $Y$ is as in (b) above;
\item $d(Y)$ may be equal to $-1,0$ or $1$ if $Y$ is as in (a) above.
\end{enumerate}
\end{lem}

\begin{proof}
Start by noticing that by \cite[Prop 1]{est1} $NF(Y)$ must be properly contained in $Y\cap \overline{Y^c}$, otherwise $I$ would be decomposable, and therefore not simple.
This immediately implies that for $Y$ as in (b), $d(Y)=\deg I_Y$. The fact that it must be equal to zero follows immediately from the definition of semistability with respect to $\E_x$, since $q_Y^{E_x}=0$ by hypothesis and $\frac{k_Y}2=\frac 12$.

Analogously, for $Y$ as in (a), we have that $k_Y=2$ hence $\deg I_Y$ may be equal to $-1,0, 1$. Therefore, since $\sharp NF(Y)=0$ or $1$, we get that $d(Y)$ could be $-1,0,1,2$. However, $d(Y)$ can not be equal to $2$ as this would imply that the semistability condition would be violated for $Y^c$ (observe that $d=d(Y)+\deg I_{\overline{Y^c}}$).
\end{proof}

We will now use Lemma \ref{L:deg-NF} in order to show that for polarisations $\E_x$ as above, $\pi_x$  lands  in $\JgP$.

\begin{prop}\label{P:forg-torsfree}
Let $(\xymatrix{ X\ar[r]_{f} & T \ar @/_/[l]_{\sigma_i,\sigma_x}}, I)$ be a section of $\JgPx^{\E_x,\star}$, where $\E_x$ is a $d$-polarisation on $\MgPx$ satisfying condition $(*)$. Then  ${\pi_x}_*(I)$ is a torsion-free rank-1 simple sheaf on $\pi_x(X\to T)=\bar X\to T$.
\end{prop}

\begin{proof}
We will start by showing that we can reduce to the case when the family $f:X\to T$ has no irreducible curves of type (b) in its fibers. Let $Y$ be such a fiber. Then, since $I$ is simple and $\E_x$-semistable, it is free at the unique node of $Y\cap \overline{Y^c}$ and $\deg I_Y=0$. Denote by $f':X'\to T$ the restriction of $f$ to the complementary locus of these components. Then one easily checks that ${\pi_x}_*(I)={\pi_x}_*(I_{|X'})$, which shows that we can ignore these components.

Consider now a family $f:X\to T$ with no irreducible components of type (b) in its fibers. Then $\pi_x(f):X\to  \bar X$ is a semistable modification of $f$ in the sense of \cite[sec. 3]{EP}. We would like to apply Theorem 3.1(2) in loc. cit., which would assert that ${\pi_x}_*(I)$ is rank-1 torsion-free, but since it holds only in the case when $I$ is a line bundle, we need to modify the family first.

Consider the family $$P(f):\mathbb P_X(I):=Proj(\oplus_{n\geq 0} Sym^n(I))\to T.$$
$P(f)$ is obtained from $X\to T$ by bubbling the nodes of the fibers of $f$ where $I$ is not free (see \cite[Prop. 5.2]{EP}).
According to Proposition 5.5 in loc. cit. the family $ P(f):X(I)\to T$ is endowed with a T-morphism  $\alpha:\mathbb P_X(I)\to X$ which is
  an isomorphism away from the bubbled exceptional components, and with a tautological line bundle $L:=\mathcal O_{\mathbb P_X(I)}(1)$ such that $\deg_EL=1$ for all exceptional components $E$ of $\alpha$ and with $\alpha_*(L)=I$.
 Consider now the semistable modification $\pi_x\circ \alpha:\mathbb P_X(I)\to \bar X$. By Lemma \ref{L:deg-NF}, the restriction of $L$ to any chain of rational components in any fiber of $P(f)$ contracted by $\pi_x\circ\alpha$ is equal to $-1$, $0$ or $1$, which, by Theorem 3.1(2) in \cite{EP}, implies that ${\pi_x}_*(I)={\pi_x\circ\alpha}_*(I)$ is torsion-free and has rank-1.

Finally, observe that since $I$ is simple, ${\pi_x}_*(I)$ remains certainly simple: it is enough to observe that a decomposition of ${\pi_x}_*(I)$ would induce a decomposition of $I$ as well and use Proposition 1 of \cite{est1}.

\end{proof}

Let now $\E_x=\left (\omega^s(\sum_{i\in A\cup\{x\}}a_i\sigma_i)\otimes \sum_{(b,B)}\alpha_{(b,B)}\da\right )\oplus \mathcal O_{\CgPx}^{r-1}$ be a $d$-polarisation on $\MgPx$ satisfying condition $(*)$. Then, since  that $a_x$ has to be equal to $0$, we can consider
$$\E:=\left (\omega^s(\sum_{i\in A}a_i\sigma_i)\otimes \sum_{(b,B)}\alpha_{(b,B)}\da\right )\oplus \mathcal O_{\CgP}^{r-1},$$
which will be a $d$-polarisation in $\MgP$.

\begin{prop}\label{P:forg-stability}
Notation as above. Given $(\xymatrix{
X\ar[r]_{f} & T \ar @/_/[l]_{\sigma_i,\sigma_x}
}, I)\in \JgPx^{\E_x,\star}(T)$, we have that $(\xymatrix{
\bar X\ar[r]_{\bar f} & T \ar @/_/[l]_{\bar \sigma_i}
}, {\pi_x}_*(I))\in\JgP^{\E,\star}(T)$.
\end{prop}

\begin{proof}
By Proposition \ref{P:forg-torsfree} above, we already know that $(\xymatrix{
\bar X\ar[r]_{\bar f} & T \ar @/_/[l]_{\bar \sigma_i}
}, {\pi_x}_*(I))\in\JgP^{\E,\star}(T)$, so it remains to show that ${\pi_x}_*(I)$ is semistable (resp. stable, resp. $\sigma$-quasistable).

It suffices to assume that $\bar X$ is a curve over $\spec k$. Let $\bar Z\subseteq \bar X$ be a subcurve of $\bar X$ and let $Z:=\pi_x^{-1}(Z)$. Then it is easy to see that
 \begin{itemize}
 \item $g_{\bar Z}=g_Z$;
 \item $k_Z=k_{\bar Z}$;
 \item $\sum_{p_i\in {\bar Z}}a_i=\sum_{p_i\in {Z}}a_i$;
 \item $\sum_{(b,B)}\alpha_{(b,B)}\delta_{b,B}^{\bar Z}=\sum_{(b,B)}\alpha_{(b,B)}\delta_{b,B}^{Z}$;
\end{itemize}
which implies that $q_Z^E=q_{\bar Z}^{E_x}$. This together with the observation that $\deg({\pi_x}_*(I))_{\bar Z}=\deg I_Z$ implies that inequality \eqref{semistabledefdegree} for ${\pi_x}_*(I)$ in $\bar Z$ is equivalent to inequality \eqref{semistabledefdegree} for $I$ in $Z$, and we conclude that ${\pi_x}_*(I)$ is $\E_x$-semistable (resp. stable, resp. $\sigma$-quasistable) if and only if $I$  is $\E$-semistable (resp. stable, resp. $\sigma$-quasistable).

\end{proof}

\subsection{Clutching compactified universal Jacobians}

The aim of the present section is to discuss when the stability conditions we have been imposing on our sheaves are compatible with the clutching morphisms  discussed in Proposition \ref{P:clutch-forget}.

\begin{prop}\label{P:clutching1}
Let
$\E=\left (\omega^s(\sum_{i\in A\cup\{x,y\}}a_i\sigma_i)\otimes \sum_{(b,B)}\alpha_{(b,B)}\da\right )\oplus \mathcal O_{\ov{\mathcal C}_{g-1,A\cup\{x,y\}}}^{r-1}$
be a $d$-polarisation on $\ov{\mathcal M}_{g-1,A\cup\{x,y\}}$ such that $a_x=a_y=s$.
Then the vector bundle $\ov\E=\left (\omega^s(\sum_{i\in A}a_i\sigma_i)\otimes \sum_{(b,B)}\alpha_{(b,B)}\da\right )\oplus \mathcal O_{\CgP}^{r-1}$ is a $(d+1)$-polarisation on $\MgP$
such that $\ov\xi_{irr}$ sends $\ov{\mathcal J}_{g-1,A\cup\{x,y\}}^{\E, s}$ (resp.  $\ov{\mathcal J}_{g-1,A\cup\{x,y\}}^{\E, ss}$, $\ov{\mathcal J}_{g-1,A\cup\{x,y\}}^{\E, \sigma}$) to $\JgP^{\ov \E,s}$ (resp. $\JgP^{\ov \E,ss}$, $\JgP^{\ov\E,\sigma}$), where $\sigma$ is any section of $\ov{\mathcal C}_{g-1,A\cup\{x,y\}}\to\ov{\mathcal M}_{g-1,A\cup\{x,y\}}$ which does not coincide with $\sigma_x$ nor with $\sigma_y$.


\end{prop}
\begin{proof}
Start by noticing that $r(\E)=r(\ov\E)=r$ and that
\begin{align*}
\deg(\ov\E)&=s(2g-2)+\sum_{i\in A}a_i\\
&=s(2(g-1)-2)+2s+\sum_{i\in A\cup\{x,y\}}a_i-a_x-a_y
\end{align*}
which, as $a_x+a_y=2s$, is equal to $\deg(\E)$. So, since
$$\deg(\E)=r(d-(g-1)+1)=r((d+1)-g+1)$$
we get that $\ov\E$ is a $(d+1)-$polarisation on $\MgP$.

Let now $(X;p_1,\dots,p_n,p_x,p_y;I)$ be a section of $\ov{\mathcal J}_{g-1,A\cup\{x,y\}}$ over an algebraically closed field $k$ and denote by $(\ov X;p_1',\dots,p_n';\ov I)$ its image under $\ov \xi_{irr}$ in $\JgP(k)$. We must show that if $(X;p_1,\dots,p_n,p_x,p_y;I)\in\ov{\mathcal J}_{g-1,A\cup\{x,y\}}^{\E, s}$ (resp.  $\ov{\mathcal J}_{g-1,A\cup\{x,y\}}^{\E, ss}$, $\ov{\mathcal J}_{g-1,A\cup\{x,y\}}^{\E, \sigma}$),  then $(\ov X;p_1',\dots,p_n';\ov I)\in\JgP^{\ov \E,s}$ (resp. $\JgP^{\ov \E,ss}$, $\JgP^{\ov\E,\sigma}$).
Notice that
$\deg(\ov I)=\deg((\ov\xi_{irr})_*(I))=\deg(I)+1=d+1.$
It remains to check that given $Y\subseteq X$,  $\deg I_Y$ satisfies (resp. strictly satisfies) inequality (\ref{semistabledefdegree}) with respect to $\E$ if and only if $\deg\ov I_{\ov Y}$ satisfies (resp. strictly satisfies) (\ref{semistabledefdegree}) with respect to $\ov\E$, where $\ov Y$ denotes the image of $Y$ under $\xi_{irr}$.

Start by considering the case when neither $p_x$ nor $p_y$ belongs to $Y$. Then we have that
 $q^\E_Y=q^{\ov\E}_{\ov Y}$
and $\deg I_Y=\deg\ov I_{\ov Y}$, which implies that inequality (\ref{semistabledefdegree}) for $Y$ with respect to $\E$ is satisfied (resp. strictly satisfied) if and only if inequality (\ref{semistabledefdegree}) for $\ov Y$ with respect to $\ov\E$ is satisfied (resp. strictly satisfied).

Suppose now that $p_x\in Y$ and $p_y\notin Y$ (the case $p_x\notin Y$ and $p_y\in Y$ is analogous).
Then $g_Y=g_{\ov Y}$, $k_Y=k_{\ov Y}-1$ and $\deg I_Y=\deg\ov I_{\ov Y}$. But then, using the fact that $a_x=s$, we easily see that
$$q^{\ov \E}_{\ov Y}=\frac{s(2g_{Y}-2+k_Y+1)+(\sum_{p_i\in Y}a_i)-a_x}{r}+\frac{\omega_Y+1}{2}=q^\E_Y+\frac 12,$$
and thus $q^{\ov \E}_{\ov Y}-\frac{k_{\ov Y}}{2}=q^{\E}_{ Y}-\frac{k_{Y}}{2}$.
Again, we conclude that inequality (\ref{semistabledefdegree}) for $Y$ with respect to $\E$ is satisfied (resp. strictly satisfied) if and only if inequality (\ref{semistabledefdegree}) for $\ov Y$ with respect to $\ov\E$ is satisfied (resp. strictly satisfied).

Finally, let us consider the case when both $p_x$ and $p_y$ are in $Y$. Then $g_Y=g_{\ov Y}-1$, $k_Y=k_{\ov Y}$ and $\deg I_Y=\deg\ov I_{\ov Y}-1$. We get the same conclusion as in the cases above since
\begin{align*}
q^\E_Y&=\frac{s(2g_Y-2+k_Y)+\sum_{p_i\in Y}a_i}{r}+\frac{\omega_Y}2\\
&=\frac{s(2(g_{\ov Y}-1)-2+k_{\ov Y})+\sum_{p_i\in \ov Y}a_i+a_x+a_y}{r}+\frac{\omega_{\ov Y}}2-1\\
&=\frac{s(2g_{\ov Y}-2+k_{\ov Y})-2s+\sum_{p_i\in \ov Y}a_i+a_x+a_y}{r}+\frac{\omega_{\ov Y}}2-1
\end{align*}
which, as $2s=a_x+a_y$, is equal to $q^{\ov \E}_{\ov Y}-1$.

The whole statement follows.
\end{proof}

The statement of the Proposition we have just proved also holds for the canonical polarisation, as we show in the following

\begin{prop}
Let $\E^{d,\un a}$ be the canonical polarisation associated to $(d,\un a)$ on $\ov{\mathcal M}_{g-1,A\cup\{x,y\}}$ and assume that $a_x=a_y=1$.
Then the clutching morphism  $\ov\xi_{irr}$ sends $\ov{\mathcal J}_{g-1,A\cup\{x,y\}}^{\E^{d,\un a}, s}$ (resp.  $\ov{\mathcal J}_{g-1,A\cup\{x,y\}}^{\E^{d,\un a}, ss}$) to $\JgP^{\E^{d+1,\un{a'}},s}$ (resp. $\JgP^{\E^{d+1,\un{a'}},ss}$), where $\E^{d+1,\un{a'}}$ is the canonical polarisation on $\MgP$ associated to $(d+1,\un{a'})=(d+1, a_i\in A\setminus\{x,y\})$.
\end{prop}

\begin{proof}
The proof is completely analogous to the proof of Proposition \ref{P:clutching1} using inequality (\ref{canonical}) and is therefore left to the reader.
\end{proof}

\begin{prop}
Let $\E_1=\omega^s(\sum_{i\in A_1\cup\{x\}}a_i\sigma_i)
\oplus \mathcal O^{r-1}$ be a $d_1$-polari\-za\-tion of rank $r$
on $\ov{\mathcal M}_{g_1,A_1\cup\{x\}}$
and $\E_2=\omega^s(\sum_{i\in A_2\cup\{y\}}a_i\sigma_i)
\oplus \mathcal O^{r-1}$ be a $d_2$-polari\-za\-tion of rank $r$
on $\ov{\mathcal M}_{g_2,A_2\cup\{y\}}$ such that $a_x=a_y=s$. Then
the vector bundle ${\E}:=\omega^s(\sum_{i\in A}a_i\sigma_i)
\oplus \mathcal O^{r-1}$ is a $d$-polarisation of rank $r$
 on $\ov{\mathcal M}_{g,A}$, where $g:=g_1+g_2$, $d=d_1+d_2+1$ and $A:=A_1\cup A_2$. Moreover,
 there are natural morphisms
 $$\ov{\mathcal J}_{g_1,A_1\cup\{x\}}^{\E_1, ss}\times \ov{\mathcal J}_{g_2,A_2\cup\{y\}}^{\E_2, ss}
 \to \JgP^{\ov \E,ss}$$
 and
 $$\ov{\mathcal J}_{g_1,A_1\cup\{x\}}^{\E_1, \sigma}\times \ov{\mathcal J}_{g_2,A_2\cup\{y\}}^{\E_2, ss}
 \to \JgP^{\ov \E,\sigma}$$
 where $\sigma$ is any section on $\ov{\mathcal M}_{g_{1},A_1\cup\{x\}}$.
\end{prop}

\begin{proof}
Start by noticing that $r(\E)=r=r(\E_1)=r(\E_2)$ and that, as $a_x=a_y=s$,
\begin{align*}
\deg(\E)&=s(2(g_1+g_2)-2)+\sum_{i\in A_1\cup A_2}a_i\\
&=s(2g_1-2+2g_2-2)+2s+\sum_{i\in A_1\cup \{x\}}a_i +\sum_{i\in A_2\cup \{y\}}a_i -a_x-a_y\\
&=\deg(\E_1)+\deg(\E_2)=r(d_1-g_1+1+d_2-g_2+1)\\
&=r(d-g+1).
\end{align*}

Let now $(X_1;p_{i, i\in A_1\cup \{x\}};I_1)$ and $(X_2;p_{i,i\in A_2\cup \{y\}};I_2)$ be sections of $\ov{\mathcal J}_{g_1,A_1\cup\{x\}}$ and of
 $\ov{\mathcal J}_{g_2,A_2\cup\{y\}}$, respectively, over an algebraically closed field $k$. Let $( X;p_{i, i\in A};I)\in \JgP(k)$ denote the image of $(X_1;p_{i, i\in A_1\cup \{x\}};I_1(p_x))$ and $(X_2;p_{i, i\in A_2\cup \{y\}};I_2)$ under $\ov\xi_{(g_1,A_1)}$. Recall that $I$ is the unique simple torsion-free sheaf on $X$ such that
 $I_{|X_1}\cong I_1(p_x)$ and $I_{|X_2}\cong I_2$.
 Then
 $$\deg(I)=\deg(I_1)+1+\deg(I_2)=d.$$

It remains to check that if $I_j$ is semistable with respect to $\E_j$ in $X_j$, for $j=1,2$,
then $I$ is semistable with respect to $\E$ in $X$.
Let $Y\subseteq X$ and denote by $Y_j$ the preimages of $Y$ under $\xi_{(g_1,A_1)}$ (notice that $Y_1$ or $Y_2$ might be empty).

Start by considering the case when $Y_2=\emptyset$ and $Y_1$ does not contain $p_x$.
 Then $g_{Y_1}=g_Y$, $k_{Y_1}=k_Y$ and $\deg_{Y_1}(I_1)=\deg_Y(I)$, which implies that inequality (\ref{semistabledefdegree}) (resp. strict inequality) for $Y_1$ with respect to $\E_1$ is equivalent to inequality (\ref{semistabledefdegree}) (resp. strict inequality) for $Y$ with respect to $\E$. The case when $Y_1=\emptyset$ and $Y_2$ does not contain $p_y$ is analogous.

Suppose now that $Y_2=\emptyset$ and that $Y_1$ contains $p_x$. Then $g_{Y_1}=g_Y$ and $k_{Y_1}=k_Y-1$ and, as $a_x=s$,
\begin{align*}
q^{\E_1}_{Y_1}&=\frac{s(2g_{Y_1}-2+k_{Y_1})+\sum_{i\in Y_1}a_i}{r}+\frac{\omega_{Y_1}}{2}\\
& =\frac{s(2g_{Y}-2+k_{Y}-1)+\sum_{i\in Y}a_i+a_x}{r}+\frac{\omega_{Y}}{2}-\frac 12\\
&=q^\E_Y-\frac 12.
\end{align*}
Since inequality (\ref{semistabledefdegree}) holds for $Y_1$ with respect to $\E_1$, we get that
$$\deg_Y(I)=\deg_{Y_1}(I_1)+1\geq q_{Y_1}^{\E_1}-\frac{k_{Y_1}}{2}+1=q_{Y}^{\E}-\frac{k_{Y}}{2}+1,$$
so
inequality (\ref{semistabledefdegree}) holds strictly for $Y$ with respect to $\E$.
The case when $Y_1=\emptyset$ and $Y_2$ contains $p_y$ is analogous.

Finally, let us consider the case when both $Y_1$ and $Y_2$ are non-empty. Then, since it is enough to consider the case when $Y$ is connected, we can assume that $p_x\in Y_1$ and $p_y\in Y_2$.
Then, we have that $g_Y=g_{Y_1}+g_{Y_2}$ and that $k_Y=k_{Y_1}+k_{Y_2}$. So, for $j=1,2$,
 inequality (\ref{semistabledefdegree}) for $Y_j$ with respect to $\E_j$ gives that
$$\deg_{Y_j}(I_j)\geq \frac{s(2g_{Y_j}-2+k_{Y_j})+\sum_{i\in Y_j}a_i}{r}+g_{Y_j}-1$$
which implies that
\begin{align*}
&\deg_{Y_1}(I_1)+\deg_{Y_2}(I_2)\geq \\
&\frac{s(2g_{Y_1}-2+k_{Y_1}+2g_{Y_2}-2+k_{Y_2})+\sum_{i\in Y_1\cup Y_2}a_i}{r} +g_{Y_1}-1+g_{Y_2}-1\\
&=\frac{s(2g_Y-2+k_Y)-2s+\sum_{i\in Y}a_i+a_x+a_y}{r}+g_Y-1-1\\
&=q_Y^{\E}-\frac{k_Y}{2}-1
\end{align*}
Since
$$\deg_Y(I)=\deg_{Y_1}(I_1)+\deg_{Y_2}(I_2)+1,$$
we get that inequality (\ref{semistabledefdegree}) holds for $Y$ with respect to $E$. Moreover, if (\ref{semistabledefdegree}) holds strictly for $Y_1$ with respect to $\E_1$, we also get strict inequality for $Y$ with respect to $\E$. The whole statement follows.

\end{proof}

\subsection{Sections from $\MgP$}

Let $\un d=(d_1,\dots,d_{|A|})$ be a tuple of numbers summing up to $d$ and consider the Abel-Jacobi map
\begin{align*}
\F_{\un d}&:\;\;\;\mathcal M_{g,A}\;\;\;\;\longrightarrow \;\;\;\;\mathcal J_{g,A}\\
&{(\xymatrix{
X\ar[r]_{f} & T \ar @/_/[l]_{\sigma_i}
})}\mapsto (\xymatrix{
X\ar[r]_{f} & T \ar @/_/[l]_{\sigma_i}
}, \mathcal O_X(\sum_{i\in A}d_i\sigma_i)).
\end{align*}

Motivated by a question posed by Eliashberg (see \ref{S:Eliashberg} below), R. Hain has asked in \cite{hain} if one could extend this map over all $\MgP$ with target a suitable universal compactified Jacobian.  We will give a positive answer to this question in what follows.

Consider the vector bundle
$$\E=\omega^{-1}(\sum_{i=1}^{|A|}2d_i \sigma_i)\oplus \O,$$
which yields a $d$-polarisation on $\MgP$ since $\deg \E=-(2g-2)+\sum_{i}2d_i=2(d-g+1)$.
Then there is a morphism
$$\F_{\un d}:\MgP\to \JgP$$
given by assigning to the section $(f:X\to T, \sigma_1,\dots, \sigma_{|A|})\in \MgP$ the same family of pointed curves endowed with the line bundle $\I=\O_X(\sum_{i\in A}d_i\sigma_i)$.
Then $\I\in\JgP^{\E,s}$ since $\deg(\I)=d$ and for all $t\in T$ and $Y\subseteq X(t)$ with $(\sigma_1(t),\dots,\sigma_{|A|}(t))=(p_1,\dots,p_{|A|})$, we have that
\begin{align*}
q^E_Y&=\frac{\deg_YE}2+\frac{w_Y}{2}\\
&=\frac{-2\omega_Y+\sum_{p_i\in Y}2d_i}{2}+\frac{w_Y}{2}\\
&=\sum_{p_i\in Y}d_i
\end{align*}
which yields that
$$\deg_Y(\I)=\sum_{p_i\in Y}d_i>\sum_{p_i\in Y}d_i -\frac{k_Y}2= q^E_Y-\frac{k_Y}2,$$
i.e., $\I$ is $\E$-stable.

Summing up we have:
\begin{prop}
Let $\un d=(d_1,\dots,d_{|A|})$ be a tuple of numbers summing up to $d$ and $\E$ the $d$-polarisation on $\MgP$ given by $\E=\omega^{-1}(\sum_{i\in A}2d_i \sigma_i)\oplus \O$.
Then there is a section
$$\F_{\un d}:\MgP\to \JgP^{\E,s}$$
given by sending a section $(f:X\to T, \sigma_1,\dots, \sigma_{|A|})\in \MgP(T)$ to $(f:X\to T; \sigma_1,\dots, \sigma_{|A|};\\
 \O_X(\sum_{i\in A}d_i\sigma_i))$.
\end{prop}

We will call the above morphism $\F_{\un d}$ an \textit{Abel-Jacobi map}.

\subsection{Applications and future work}\label{S:applications}

The compactified universal Jacobian stacks we have constructed satisfy, as we saw, a number of nice properties as the existence of clutching morphisms, forgetful and section morphisms from $\MgP$, the existence of a projective coarse moduli space and many others.
Therefore, there is a number of questions that appear naturally connected to different possible applications of these constructions as we indicate here below.
For instance, given the nice modular description of our compactifications, it is natural to study their intersection theoretical properties, and use their forgetful and section morphisms to $\MgP$ to describe certain tautological classes in $\MgP$. This is the idea that stands behind the following question.

\subsubsection{The Eliashberg problem}\label{S:Eliashberg}
Assume now that $d=0$. Then the universal Jacobian stack $\Jac^0$ is a dense open substack of $\JgP^{\E,s}$ and the restriction of above morphism
$\F_{\un d}:\MgP\to \JgP^{\E,s}$ to $\mathcal M_{g,A}$, which we will also denote with $\F_{\un d}$, lands in $\Jac^0$.
Consider the locus $R_{\un d}$ of curves $(X;p_i\in A)\in \mathcal M_{(g,A)}$ such that $\sum_{p_i\in A}d_i p_i$ is a principal divisor on $X$. This locus can be seen as the locus of curves admitting a map to $\mathbb P^1$ with prescribed preimages and ramification over two points, say $0$ and $\infty$, so it is a natural ``double Hurwitz'' locus which is known as the \textit{double ramification locus}. Consider the following well-known question, which is due to Eliashberg:
\begin{question}
What is the class of the closure $\ov R_{\un d}$ of $R_{\un d}$ in $\MgP$?
\end{question}

Let $Z_{g,A}$ be the zero section of $\Jac$. Then the divisor $\sum_{p_i\in A}d_i p_i$ is principal if and only if its image under $\F_{\un d}$ lands in $Z_{g,A}$.
Let $\ov Z_{g,A}$ be the closure in $\JgP^{\E,s}$ of the zero section $Z_{g,A}$ of $\Jac$. Then the class of an extension of $R_{\un d}$ in $\MgP$ can be computed by pulling back $\ov Z_{g,A}$ under $\F_{\un d}$. This computation has been done by Hain in \cite{hain} over the locus $\mathcal M_{g,A}^{ct}$ of curves of compact type using the fact that the Torelli morphism extends over that locus. More recently, this computation was further extended by Grushevsky and Zakharov in \cite{gru} for the locus $\ov{\mathcal M}_{g,A}^o$ of curves whose normalization has genus at least $g-1$ and using the fact that the Torelli morphism extends over that locus to a partial compactification of the moduli space of principally polarized abelian varieties $\mathcal A_g$. It is natural to ask the following
\begin{question}
Is it possible to compute the class of the closure of the zero section $\ov Z_{g,A}$ on $\JgP^{\E,s}$ and then pull it back to $\MgP$ via $\mathcal F_{\un d}$ to compute the class of  $\ov R_{\un d}$?
\end{question}

In fact, in \cite{dudin}, B. Dudin used the existence and properties of our universal compactified Jacobians $\JgP$ in order to give a partial answer to the above question. More precisely, he considered a polarisation $\E$ such that the zero section of $\mathcal J_{g,A}^0$ extends over the whole $\JgP^{\E,ss}$ and he constructed an extension of $F_{\un d}$ over a locus $\MgP^\tau$ on $\MgP$
containing the locus of tree-like curves. Then he computed the class of the zero section on $\JgP^{\E,s}$ and its pullback to $\MgP^\tau$, extending further the works of Hain and of Grushevsky-Zhakarov.
We hope that by slightly modifying this argument we can actually get to a complete answer to the Eliashberg question.

A different approach to compute a (possibly different) completion of $R_{\un d}$ was described in work of Cavalieri, Marcus and Wise in \cite{cavalieri} and consists of computing the pushforward of the virtual fundamental class of the moduli space of relative stable maps to a rubber $\mathbb P^1$ by the forgetful map to $\MgP$. A. Pixton conjectured a formula for this class and this conjecture was recently proved to be true by work of Janda, Pandharipande, Pixton and Zvonkine in \cite{DRC}.
Both this and the ``Jacobian type'' completion described by Hain coincide over the locus of curves of compact type as shown by Marcus and Wise in \cite{MW}. We actually believe the same proof as in loc. cit. suffices to show that both completions coincide over the locus of tree-like curves. It would certainly be very interesting to understand the relation between the two in a wider locus of $\MgP$; we intend to go back to this problem in a near future.
For an account on the different perspectives  on the subject (including Pixton's conjecture) see \cite{cav}.

\subsubsection{Generalized ELSV formulas}\label{ELSV}

The so-called ELSV-formula establishes a surprising relation between certain intersection numbers on the moduli space of stable marked cur\-ves with ``one part double Hurwitz numbers'', which count the number of coverings of the projective line $\mathbb P^1$ by curves of genus $g$ with assigned ramification over $\infty$ and simple ramification in the remaining ramification points (see \cite{ELSV1}).
It is conjectured by Goulden, Jackson and Vakil (see \cite{GJV}) that there should exist a compactification of the universal Picard stack over the moduli space of stable curves with marked points where the intersection theory is strictly related with the geometry of ``one part double Hurwitz numbers'': the number of coverings of $\mathbb P^1$ having prescribed ramification over $\infty$, total ramification over $0$ and simple ramification in the remaining ramification points.
 More
precisely, denote by $d$ the degree of the cover, by  $\beta=(\beta_1,\dots,\beta_n)$ the ramification type over infinity
and by $g$ the genus of the curve. Then the covers of $\mathbb P^1$ with these numerical invariants should be
proportional to
 $$\int_{\overline{\mathcal J}_{g,n}}\frac{\Lambda_1+\dots\pm \Lambda_{2g}}{(1-\beta_1\psi_1)+\dots+(1-\beta_n\psi_n)}$$
 where the classes $\psi_i$ should be a natural generalization of the classes $\psi_n$ on $\mgnb$ and the $\Lambda_k$'s
should be suitable Chow classes of codimension $2k$ associated to a certain autodual rank $2g$ vector
bundle. If this formula would be true, it would give a natural and beautiful generalization of
the ELSV-formula, so it is natural to try to investigate if one can give an answer to this conjecture. This would pass by studying intersection theory in our compactified universal Jacobian stacks and then by applying a virtual localization argument following the lines of Graber and Vakil's proof of the ELSV formula in \cite{graber}.

Double Hurwitz numbers in general (rather than just one part double Hurwitz numbers) were shown in \cite{GJV} to be piecewise polynomial. Later,  in \cite{CJM1} and \cite{CJM2}, Cavalieri, Markwig and Johnson have studied their chamber structure along with wall crossing formulas by relating them to certain tropical objects.
In the light of Goulden Jackson and Vakil's conjecture, it would be interesting to investigate if by varying the polarisation we can describe the chamber structure and wall crossing formulas for our universal compactified Jacobians and establish an analogy with the phenomena described by Cavalieri, Markwig and Johnson for double Hurwitz numbers.

Recall that in \cite{KP}, the authors have described a compactification of the universal Jacobian of degree $g-1$ over the locus of marked tree-like curves depending upon a stability parameter along with a wall crossing formula for the theta divisor as the stability parameter varies, wich is suggestive of the wall-crossing phenomenon discussed in the previous paragraph. As we have shown in Proposition \ref{P:compare}, all the compactifications in loc. cit. can be realised by restricting our compactified universal Jacobian for some polarisation to the locus of tree-like curves. It is therefore expected that we can obtain similar wall crossing formulas for other Chow classes in our compactifications as the polarisation varies over the whole $\MgP$.

\subsubsection{Gromov-Witten invariants for $\rm B\mathbb G_m$}\label{GW}

Algebro-geometric Gromov-Witten invariants for a smooth projective variety $X$ are defined using the Kontsevich moduli stack $\overline{\mathcal M}_{g,n}(X,\beta)$ parametrizing $n$-pointed stable maps from curves of genus $g$ to $X$ with image class $\beta\in H_2(X,\mathbb Z)$. This space is endowed with evaluation maps $\rm{ev}_i$ with target on $X$ and with a virtual fundamental class.
The invariants are then defined by capping classes with the virtual fundamental class and then by pushing forward via the evaluation maps.

Gromov-Witten invariants with target on orbifolds were defined by Chen and Ruan in \cite{CR} and Abramovich, Graber and Vistoli in \cite{AGV}. There are two main differences in this situation: the first one is that the source curves are replaced with twisted curves to maintain the properness of moduli and the second one is that the evaluation maps take values on the so-called \textit{rigidified cyclotomic inertia stack} of the target orbifold.

In \cite{FTT}, Frenkel, Teleman and Tolland announced a construction of gauge theoretical Gromov-Witten invariants for $[\rm pt/\mathbb C^*]$. For that purpose they construct a modular completion of the moduli stack of stable maps from marked curves to $[\rm pt/\mathbb C^*]$ consisting of bubbled marked stable curves together with a $\mathbb C^*$-bundle. This stack is endowed with evaluation maps to $[\rm pt/\mathbb C^*]$ and with a forgetful morphism $\Pi$ onto the moduli stack of stable marked curves but it is far from being proper. The invariants constructed in loc. cit. are, as in the classical case, obtained by pushing a product of evaluation classes forward along the forgetful morphism $\Pi$, but since the fibers of this map are Artin stacks, the invariants are defined in K-theory rather than in cohomology.

An algebro-geometric construction of Gromov-Witten invariants for Artin stacks is, at least to our knowledge, still missing, and the first case to consider is naturally the case of $\rm B\mathbb G_m$.
The universal compactified Jacobian stacks we construct are smooth and proper Deligne-Mumford stacks which can be seen as a modular completion of the moduli stack parametrizing stable marked curves together with a line bundle, i.e., a map to $\rm B\mathbb G_m$. Fix then a polarisation $\E$ and a section $\sigma$ of $\pi_x:\MgPx\to\MgP$ and set $\overline \M:=\JgP^{\E,\sigma}$. Then $\overline\M$ is endowed with $n$ evaluation maps $\rm ev_j: \overline \M\to \rm B\mathbb G_m$ defined by
assigning to a section $(\xymatrix{
X\ar[r]_{f} & T \ar @/_/[l]_{\sigma_i}
},I)$ of $\ov\M(T)$, the section of $\rm B\mathbb G_m(T)$ given by $\sigma_j^*(I)$ (notice that this is a principal $\mathbb G_m$-bundle over $T$ because the image of the sections $\sigma_i$ is contained in the smooth locus of $X$).
Now, as we already observed, the stacks $\ov\M$ are smooth and proper, so there is no need for establishing the existence of a virtual fundamental class. In order to define the invariants one should now consider the pushforward of classes on $\ov\M$ onto $\rm B\mathbb G_m$.
We hope to come back to this in a near future.



\end{document}